\documentclass[11pt]{article}

\usepackage[T1]{fontenc}
\usepackage{lmodern}
\usepackage[a4paper,margin=1in]{geometry}
\usepackage{amsmath,amssymb,amsthm,mathtools}
\usepackage{bm}
\usepackage{graphicx}
\usepackage{booktabs}
\usepackage{array}
\usepackage{longtable}
\usepackage{xcolor}
\usepackage{enumitem}
\usepackage{tikz}
\usetikzlibrary{positioning,arrows.meta,calc}
\usepackage[authoryear,round]{natbib}
\usepackage[colorlinks=true,allcolors=mHeaderBlue]{hyperref}
\usepackage[capitalize,noabbrev]{cleveref}
\usepackage{placeins}
\usepackage[protrusion=true,expansion=false]{microtype}
\definecolor{mHeaderBlue}{HTML}{4057A8}
\definecolor{proved}{HTML}{2E7D5B}
\definecolor{cond}{HTML}{B7791F}
\definecolor{conj}{HTML}{B0413E}

\newtheorem{theorem}{Theorem}[section]
\newtheorem{proposition}[theorem]{Proposition}
\newtheorem{lemma}[theorem]{Lemma}
\newtheorem{corollary}[theorem]{Corollary}
\newtheorem{conjecture}[theorem]{Conjecture}
\theoremstyle{definition}
\newtheorem{definition}[theorem]{Definition}
\newtheorem{remark}[theorem]{Remark}

\newtheorem{principle}[theorem]{Principle}

\newcommand{\tier}[2]{{\small\textsf{\textcolor{#1}{[#2]}}}}
\newcommand{\Proved}{\tier{proved}{Proved}}
\newcommand{\Conditional}{\tier{cond}{Conditional}}
\newcommand{\Conjectural}{\tier{conj}{Conjectural}}
\newcommand{\R}{\mathbb{R}}
\newcommand{\N}{\mathbb{N}}

\newcommand{\E}{\mathbb{E}}

\newcommand{\G}{\mathcal{G}}

\newcommand{\dd}{\,\mathrm{d}}
\newcommand{\Ito}{It\^o}

\newcommand{\cadlag}{c\`adl\`ag}

\newcommand{\Sig}{\mathrm{Sig}}

\newcommand{\Tr}{\operatorname{Tr}}

\newcommand{\one}{\mathbf 1}
\newcommand{\emptyword}{\varnothing}
\newcommand{\norm}[1]{\left\lVert #1\right\rVert}

\newcommand{\inner}[2]{\left\langle #1,#2\right\rangle}
\newcommand{\shuffle}{\mathbin{\sqcup\!\sqcup}}

\newcommand{\expt}{\exp_{\otimes}}
\newcommand{\oprod}{\overrightarrow{\textstyle\prod}}
\newcommand{\lie}{\mathfrak{L}}

\DeclareMathOperator{\Cov}{Cov}
\DeclareMathOperator{\Var}{Var}
\title{\bfseries A General Theory of Paths:\\
\large Signatures, Jump Lifts, and Expected Signatures of Self-Exciting Processes}

\author{Miquel Noguer i Alonso\\ \\ Artificial Intelligence Finance Institute}

\date{\today}

\begin{document}
\maketitle

\begin{abstract}
This paper develops a path-first theory in which the signature is used as a universal coordinate for deterministic paths, rough paths, jump streams, and path-valued random variables. The organizing thesis is that geometricity is a first-order algebraic property whose obstructions are second order: a bracket when one chooses a non-geometric lift, and a covariance when one averages over random paths. This principle links the shuffle identity, the Marcus--\Ito{} distinction, expected signatures, signature kernels, and the geometry of the free nilpotent group.

The paper has four central contributions. First, the geometricity-defect theorem identifies quadratic covariation and coordinate covariance as the two canonical failures of shuffle multiplicativity, implying that an expected signature is group-like only for a deterministic reduced path. Second, the Hopf square proves that for pure-jump finite-variation paths the forward \Ito{} signature is exactly the iterated-sums signature, while the Marcus signature is Hoffman's exponential image of it; thus jump convention and discrete-to-continuous conversion are the same algebraic transition. Third, affine and exponential Hawkes processes are shown to admit finite-dimensional linear closures for truncated expected signatures after state-weight augmentation; for scalar Hawkes clocks the level-two matrix is written explicitly and the first expected-signature coordinate locally identifies the baseline, excitation, and decay parameters. Fourth, an antisymmetric second-level cross-area is proved to detect two-channel Hawkes excitation direction to first order.

Secondary consequences are included as supporting material: kernel-MMD decompositions, the central tower of free nilpotent truncations, antipode reversal of cross-area, the stable-law moment threshold, normalized expected signatures for heavy-tailed regimes, and a contraction-based signature large-deviation principle. A reproducibility script validates the algebraic identities, Hawkes formulas, Hopf square, Marcus--\Ito{} gap, identifiability reconstruction, and cross-area sign experiment.
\end{abstract}

\newpage
\tableofcontents

\section{Introduction}
\label{sec:intro}

A stochastic process can be viewed as a probability law on a space of paths. This observation is universal but incomplete: it says where the law lives, not what coordinate system should be used to analyze the path itself. The signature provides such a coordinate system. For a bounded-variation path $x:[0,T]\to E$, it is the tensor series
\begin{equation}
\Sig(x)_{0,T}
=
\left(1,
\int_{0<u_1<T}\dd x_{u_1},
\int_{0<u_1<u_2<T}\dd x_{u_1}\otimes\dd x_{u_2},
\ldots\right),
\end{equation}
with suitable rough-path extensions when the path is not classically integrable. The signature is multiplicative under concatenation, obeys the shuffle relations, and is faithful on reduced paths. These facts make it a natural universal coordinate for deterministic trajectories and for path-valued random variables.

The aim of this paper is to state a general theory of paths in a way that is useful both for pure path theory and for event-driven financial modeling. The organizing idea is the representation ladder
\begin{equation}
\boxed{
\text{path}
\longrightarrow
\text{enhanced path}
\longrightarrow
\text{full signature}
\longrightarrow
\text{truncated signature}
\longrightarrow
\text{expected signature}.}
\end{equation}
Each arrow forgets a precise layer of information. The full signature forgets parametrization and tree-like excursions but remains faithful on reduced paths. A truncation is a finite feature map. Taking expectation turns a realized path coordinate into a moment coordinate of the path law.

The paper is not only an exposition of known signature theory. The original contribution is a concrete expected-signature theory for self-exciting jump paths. Hawkes processes are the canonical model of clustered event arrivals. They are neither Levy processes nor independent-increment objects, so their expected signatures do not reduce to a tensor Levy--Khintchine exponential. Nevertheless, in the exponential-kernel case their intensity is Markovian, and this Markov structure closes a finite-dimensional system for every truncated expected signature. This gives a direct bridge between self-exciting stochastic clocks and algebraic path coordinates.

\paragraph{Core contributions.}
The paper is organized around four main results.
\begin{enumerate}[label=(\roman*),leftmargin=2.25em]
\item \textbf{Geometricity defect.} Bracket and covariance are shown to be the two second-order obstructions to shuffle multiplicativity. This explains simultaneously why the \Ito{} lift is non-geometric and why expected signatures usually leave the group of group-like tensors.
\item \textbf{Hopf square for jumps and streams.} For pure-jump finite-variation paths, the forward \Ito{} signature equals the iterated-sums signature and the Marcus signature equals Hoffman's exponential image. The jump-convention problem and the discrete/continuous-signature problem are therefore one Hopf-algebraic transition.
\item \textbf{Self-exciting expected signatures.} Truncated affine expected signatures close linearly after adjoining state-weighted coordinates. The scalar exponential Hawkes case is made explicit through a level-two matrix system and a local identifiability formula for $(\mu,\alpha,\beta)$ from derivatives of $T\mapsto\E[N_T]$.
\item \textbf{Directional cross-area.} A second-level antisymmetric signature coordinate detects the direction of two-channel Hawkes cross-excitation to leading order, giving a signature-native lead--lag statistic for event systems.
\end{enumerate}

\begin{figure}[t]
\centering
\includegraphics[width=0.95\linewidth]{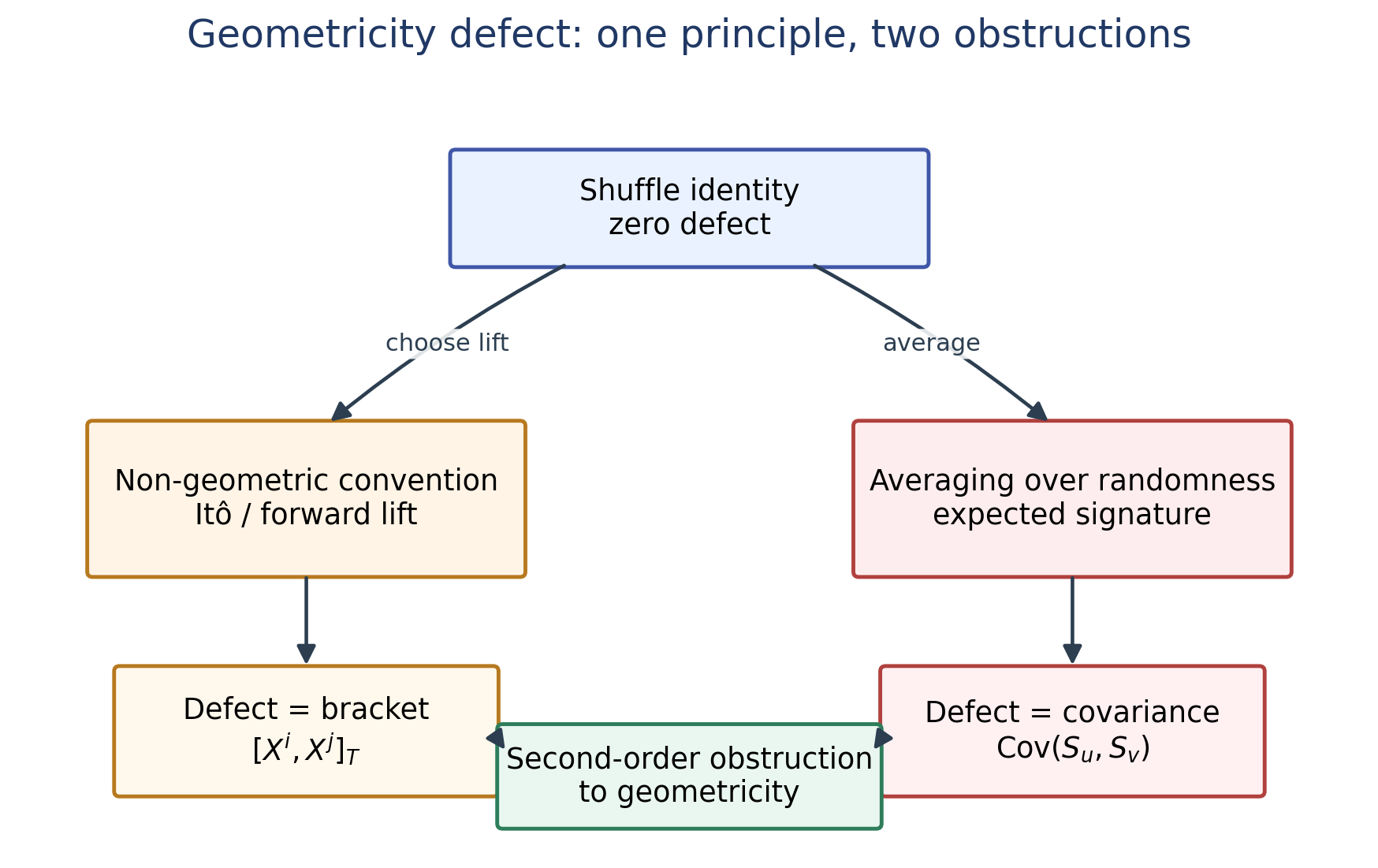}
\caption{The geometricity-defect principle. The shuffle identity is the zero-defect case. A non-geometric convention produces a bracket defect, while averaging over random paths produces a covariance defect.}
\label{fig:geometricity-defect}
\end{figure}

\paragraph{Supporting consequences.}
The remaining results clarify the scope of the theory: kernel MMDs are Euclidean distances between truncated expected signatures; the representation ladder is a central tower of free nilpotent groups; reversal is the tensor antipode and makes cross-area reversal-odd; raw expected signatures have a sharp stable-law moment threshold; normalized expected signatures remain available under characteristic-kernel hypotheses; and signature large deviations follow by contraction.

\paragraph{What is new and what is used.}
Classical inputs include Chen multiplicativity, the shuffle identity, Chen--Ree, Hambly--Lyons uniqueness, Lyons' extension theorem, and the expected-signature characterization under its usual determinacy hypotheses. The new synthesis is the geometricity-defect principle and the Hopf-square framing. The new calculations are the Hawkes expected-signature closures, the explicit level-two Hawkes matrix, local scalar Hawkes identifiability, and the leading-order cross-area direction law. Conditional or cited consequences are kept separate in the claim-tier table in \cref{app:claims}.

All theorem statements are tiered in the appendix as new proved, new synthesis, classical/cited, conditional, or conjectural.

\section{Tensor algebra and signatures}
\label{sec:tensor}

Let $E=\R^d$. The completed tensor algebra over $E$ is
\begin{equation}
T((E))=\prod_{k=0}^{\infty}E^{\otimes k},
\end{equation}
and its level-$m$ truncation is
\begin{equation}
T^{(m)}(E)=\bigoplus_{k=0}^{m}E^{\otimes k}.
\end{equation}
The empty word is denoted by $\emptyword$. For a word $w=i_1\cdots i_k$ over the alphabet $\{1,\ldots,d\}$, $|w|=k$ and $e_w=e_{i_1}\otimes\cdots\otimes e_{i_k}$. The coordinate of a tensor series $a\in T((E))$ along $w$ is $\inner{a}{w}$.

\begin{definition}[signature]
Let $x:[0,T]\to E$ be a path of bounded variation. Its signature over $[s,t]$ is
\begin{equation}
\Sig(x)_{s,t}
=
\left(1,
\int_{s<u_1<t}\dd x_{u_1},
\int_{s<u_1<u_2<t}\dd x_{u_1}\otimes\dd x_{u_2},
\ldots\right).
\end{equation}
For a word $w=i_1\cdots i_k$,
\begin{equation}
\inner{\Sig(x)_{s,t}}{w}
=
\int_{s<u_1<\cdots<u_k<t}\dd x^{i_1}_{u_1}\cdots\dd x^{i_k}_{u_k}.
\end{equation}
\end{definition}

\begin{theorem}[Chen identity]
\label{thm:chen}
Let $x$ be a bounded-variation path and $s\le u\le t$. Then
\begin{equation}
\Sig(x)_{s,t}=\Sig(x)_{s,u}\otimes\Sig(x)_{u,t}.
\end{equation}
Consequently, the signature is a monoid homomorphism from chronological concatenation of paths to tensor multiplication.
\end{theorem}

\begin{proof}
For a word $w=i_1\cdots i_k$, split the simplex
\begin{equation}
\{s<t_1<\cdots<t_k<t\}
\end{equation}
according to the number $j$ of integration variables lying in $[s,u]$. This gives
\begin{equation}
\inner{\Sig(x)_{s,t}}{w}
=
\sum_{j=0}^k
\inner{\Sig(x)_{s,u}}{i_1\cdots i_j}
\inner{\Sig(x)_{u,t}}{i_{j+1}\cdots i_k},
\end{equation}
which is exactly the coordinate formula for tensor multiplication. The concatenation statement follows by applying the same identity at the joining time.
\end{proof}

\begin{theorem}[shuffle identity]
\label{thm:shuffle}
For bounded-variation paths and words $u,v$,
\begin{equation}
\inner{\Sig(x)}{u}\,\inner{\Sig(x)}{v}
=
\inner{\Sig(x)}{u\shuffle v},
\end{equation}
where $u\shuffle v$ is the sum of all order-preserving interlacings of $u$ and $v$.
\end{theorem}

\begin{proof}
The product of the two iterated integrals is an integral over the product of two ordered simplexes. This product domain decomposes into disjoint simplexes indexed by the shuffles of the two ordered lists of integration times. Integrating over these simplexes yields exactly the displayed identity.
\end{proof}

\begin{theorem}[log-signature]
\label{thm:logsig}
The signature of a bounded-variation path is group-like. Equivalently, its logarithm belongs to the free Lie algebra over $E$.
\end{theorem}

\begin{proof}
The shuffle identity states that the coordinate functional $w\mapsto\inner{\Sig(x)}{w}$ is a character of the shuffle Hopf algebra. Characters are group-like elements in the dual tensor Hopf algebra. The equivalence between group-like elements and exponentials of primitive elements gives the free-Lie statement; this is the classical Chen--Ree theorem \citep{chen1957,ree1958,reutenauer1993}.
\end{proof}

The dimension of the free Lie algebra at level $k$ over $d$ letters is
\begin{equation}
\ell_d(k)=\frac{1}{k}\sum_{q\mid k}\mu_{\rm Mob}(q)d^{k/q},
\end{equation}
where $\mu_{\rm Mob}$ is the Mobius function.

\begin{corollary}[Witt dimensions]
\label{cor:witt}
The degree-$k$ layer $\lie_k(\R^d)$ of the free Lie algebra has dimension
\begin{equation}
\ell_d(k)=\frac{1}{k}\sum_{q\mid k}\mu_{\rm Mob}(q)d^{k/q}.
\end{equation}
Consequently the step-$N$ free Lie algebra has dimension $\sum_{k=1}^{N}\ell_d(k)$.
\end{corollary}

This explains why the log-signature is a compressed coordinate relative to the full tensor basis.

\section{Reduced paths and the representation ladder}
\label{sec:reduced}

A path followed by its exact reverse has trivial signature. This is not a defect; it says that signatures represent paths modulo a geometrically meaningful cancellation.

\begin{definition}[tree-like equivalence]
A bounded-variation path is tree-like if it factors through a real tree in such a way that the initial and terminal points in the tree coincide. Two paths are tree-like equivalent if their concatenation with one path reversed is tree-like.
\end{definition}

\begin{theorem}[signature uniqueness modulo tree-like equivalence]
\label{thm:hambly}
Two bounded-variation paths have the same signature if and only if they are tree-like equivalent. Hence the signature is injective on the reduced path group.
\end{theorem}

\begin{proof}
This is the Hambly--Lyons uniqueness theorem \citep{hambly2010}; the rough-path extension is due to \citet{boedihardjo2016}. The forward direction identifies equality of signatures with triviality of the signature of the concatenation $x*\overleftarrow{y}$; the theorem states that the only bounded-variation paths with trivial signature are tree-like. The reverse direction follows because the signature of a tree-like path is the tensor unit and Chen's identity gives cancellation.
\end{proof}

\begin{principle}[representation ladder]
\label{prin:ladder}
The path coordinate system has the projection chain
\begin{equation}
\boxed{
(t,\mathbf x_t)
\twoheadrightarrow
\Sig(\mathbf x)
\twoheadrightarrow
\Sig^{(m)}(\mathbf x)
\twoheadrightarrow
\E[\Sig^{(m)}(X)].}
\end{equation}
The first projection forgets speed and tree-like cancellation; the second forgets tensor levels above $m$; the third forgets the sample path and keeps law-level moments.
\end{principle}

\begin{table}[t]
\centering
\renewcommand{\arraystretch}{1.25}
\begin{tabular}{@{}p{0.22\linewidth}p{0.31\linewidth}p{0.36\linewidth}@{}}
\toprule
Layer & Object & Information retained \\
\midrule
Path & $(t,\mathbf x_t)$ & parametrized enhanced trajectory \\
Full signature & $\Sig(\mathbf x)$ & reduced path, chronological algebra, all iterated integrals \\
Truncated signature & $\Sig^{(m)}(\mathbf x)$ & finite path features through level $m$ \\
Expected signature & $\E[\Sig^{(m)}(X)]$ & noncommutative moment coordinates of a law \\
\bottomrule
\end{tabular}
\caption{The representation ladder. Each lower layer is a controlled loss of information.}
\label{tab:ladder}
\end{table}

\section{Rough paths, jumps, and the Marcus--\Ito{} distinction}
\label{sec:jumps}

For irregular continuous paths, rough-path theory replaces a path by a finite tower of iterated integrals satisfying Chen's identity and analytic $p$-variation bounds \citep{lyons1998,lcl2007,frizhairer2020}. For \cadlag{} paths, one must additionally choose a convention at jumps.

\begin{definition}[two jump lifts]
Let $X:[0,T]\to\R^d$ be a finite-variation \cadlag{} path. The forward or \Ito{} level-two lift is
\begin{equation}
\mathbb X^{I}_{s,t}
=
\int_{(s,t]}(X_{u-}-X_s)\otimes\dd X_u.
\end{equation}
The Marcus level-two lift is
\begin{equation}
\mathbb X^{M}_{s,t}
=
\mathbb X^{I}_{s,t}
+\frac12\sum_{s<u\le t}\Delta X_u\otimes\Delta X_u.
\label{eq:marcus-correction}
\end{equation}
Equivalently, the Marcus lift replaces each jump by a straight chord traversed over an auxiliary interval and then takes the ordinary geometric signature of the completed path.
\end{definition}

\begin{proposition}[jump correction]
\label{prop:jump-correction}
For every finite-variation \cadlag{} path,
\begin{equation}
\mathbb X^{M}_{s,t}-\mathbb X^{I}_{s,t}
=\frac12\sum_{s<u\le t}\Delta X_u\otimes\Delta X_u.
\end{equation}
The Marcus lift satisfies the geometric shuffle identity; the forward \Ito{} lift satisfies it if and only if the total jump-square correction vanishes at level two.
\end{proposition}

\begin{proof}
The first identity follows directly from the definition of the Marcus completed graph: on each jump chord of size $h$, the ordinary second iterated integral along the straight segment is $h\otimes h/2$, while the forward integral across the jump contributes zero to the diagonal self-area of that jump. Summing over jumps gives the formula. The shuffle identity at level two requires
\begin{equation}
\inner{\Sig}{i}\inner{\Sig}{j}
=
\inner{\Sig}{ij}+\inner{\Sig}{ji}.
\end{equation}
For the Marcus lift this is the ordinary integration-by-parts identity on the completed graph. For the forward lift, the missing symmetric part is precisely the jump-square sum above. Hence the forward lift is geometric at level two exactly when that correction is zero.
\end{proof}

\begin{proposition}[\Ito{} is not a shuffle character; the defect is the jump bracket]
\label{prop:ito-defect}
Let $X$ be a finite-variation \cadlag{} path in $\R^d$, and let $\Sig^I$ denote its forward lift. For all letters $i,j$,
\begin{equation}
\inner{\Sig^I}{i}\inner{\Sig^I}{j}
-\big(\inner{\Sig^I}{ij}+\inner{\Sig^I}{ji}\big)
=\sum_{0<u\le T}\Delta X_u^i\Delta X_u^j
=:[X^i,X^j]^{\mathrm d}_T.
\label{eq:ito-defect}
\end{equation}
Thus the forward lift fails the shuffle relation exactly by the discrete jump quadratic covariation. In particular, unless this bracket vanishes, $\Sig^I$ is not a character on the shuffle Hopf algebra.
\end{proposition}

\begin{proof}
The forward level-two coordinates are
\begin{equation}
\inner{\Sig^I}{ij}=\int_{(0,T]}X^i_{u-}\dd X^j_u,
\qquad
\inner{\Sig^I}{ji}=\int_{(0,T]}X^j_{u-}\dd X^i_u,
\end{equation}
after translating the initial point to zero. The finite-variation product rule for \cadlag{} paths gives
\begin{equation}
X_T^iX_T^j
=\int_{(0,T]}X^i_{u-}\dd X^j_u
+\int_{(0,T]}X^j_{u-}\dd X^i_u
+\sum_{0<u\le T}\Delta X_u^i\Delta X_u^j.
\end{equation}
Since $X_T^i=\inner{\Sig^I}{i}$ and $X_T^j=\inner{\Sig^I}{j}$, rearrangement gives \eqref{eq:ito-defect}.
\end{proof}

\begin{corollary}[Marcus restores geometricity]
\label{cor:marcus-restores}
At level two the Marcus correction
\begin{equation}
\frac12\sum_{0<u\le T}\Delta X_u\otimes\Delta X_u
\end{equation}
adds one half of the bracket to each ordering. Consequently, the Marcus lift satisfies the shuffle identity at level two, and its completed-graph construction gives a weakly geometric rough path at all levels \citep{chevyrev2019,frizshekhar2017}.
\end{corollary}

\begin{theorem}[Marcus uniqueness for nondecreasing positive-jump paths]
\label{thm:marcus-unique}
Let $x,y:[0,T]\to\R$ be finite-variation \cadlag{} paths with $x_0=y_0=0$, nondecreasing continuous parts, and strictly positive jumps. Let
\begin{equation}
\Gamma(x)_t=(t,x_t)
\end{equation}
be the time-augmented path, and let $\widehat\Gamma(x)$ denote its Marcus completed graph. If
\begin{equation}
\Sig(\widehat\Gamma(x))=\Sig(\widehat\Gamma(y)),
\end{equation}
then $x=y$ as \cadlag{} paths. Moreover, the forward \Ito{} lift of a path with at least one nonzero jump is not a geometric signature lift of the same completed graph.
\end{theorem}

\begin{proof}
The completed graph $\widehat\Gamma(x)$ is coordinatewise nondecreasing: along continuous pieces the time coordinate increases and the spatial coordinate does not decrease; along each jump chord the time coordinate is constant and the spatial coordinate strictly increases. Therefore no nontrivial subpath can be followed by an exact retracing, since retracing would require at least one coordinate to decrease. Thus $\widehat\Gamma(x)$ is tree-reduced. The same holds for $\widehat\Gamma(y)$.

By \cref{thm:hambly}, equality of signatures implies tree-like equivalence. Since both completed graphs are tree-reduced, tree-like equivalence reduces to equality up to increasing reparametrization. The first coordinate of the completed graph records the original physical time along horizontal pieces, and the vertical chords occur at the corresponding fixed time. Hence equality up to increasing reparametrization recovers the same jump times and the same spatial levels. Therefore $x=y$.

For the final assertion, if a nonzero jump $h$ occurs, \cref{prop:jump-correction} shows that the forward level-two lift differs from the geometric completed-graph lift by $h\otimes h/2$ plus the remaining jump corrections. Thus the forward lift violates the level-two shuffle identity unless the total correction vanishes. With positive jumps it cannot vanish.
\end{proof}

\begin{remark}
The monotonicity assumption is sufficient, not necessary. It covers cumulative event counts, cumulative traded volume, cumulative order-flow imbalance in a fixed direction, and default-count processes. General price paths with bid--ask bounce need the full reduced-path condition rather than coordinatewise monotonicity.
\end{remark}

\section{Discrete streams and continuous signatures}
\label{sec:discrete}

Discrete streams are represented by iterated sums, whereas continuous chord interpolations are represented by iterated integrals. Hoffman's exponential identifies these two presentations at the Hopf-algebraic level \citep{hoffman2000,diehl2020}. The key point is exact: for pure-jump paths the forward \Ito{} lift is the iterated-sums character, while the Marcus lift is its Hoffman image.

\begin{figure}[t]
\centering
\includegraphics[width=0.92\linewidth]{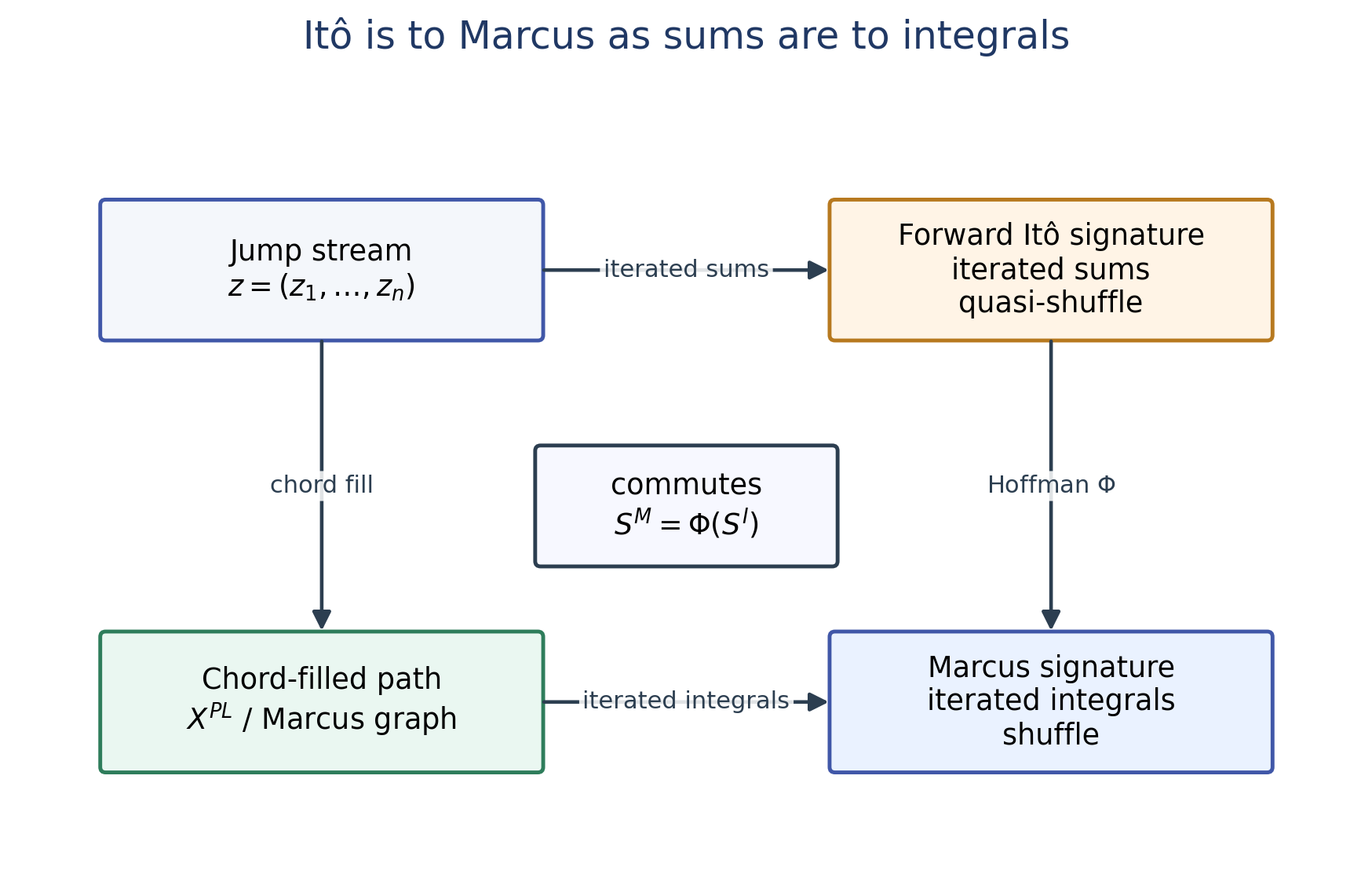}
\caption{The Hopf square. The forward \Ito{} lift of a pure-jump path is the iterated-sums signature; the Marcus chord-fill lift is Hoffman's exponential image and is therefore the shuffle-geometric version.}
\label{fig:hopf-square}
\end{figure}

\begin{definition}[Hoffman exponential]
\label{def:hoffman}
On the enlarged alphabet $\overline E=\bigoplus_{m\ge1}E^{\odot m}$ of symmetric powers, in which a merged letter $[i_1\cdots i_r]$ denotes the symmetric product, let $\Phi$ be Hoffman's exponential: the graded Hopf-algebra isomorphism from the quasi-shuffle to the shuffle algebra acting on a word $w$ by
\begin{equation}
\Phi(w)=\sum_{w=B_1\mid\cdots\mid B_j}\frac{1}{|B_1|!\cdots|B_j|!}\,[B_1]\cdots[B_j],
\label{eq:hoffman-exp}
\end{equation}
the sum ranging over compositions of $w$ into consecutive blocks. On the jump data a merged letter is evaluated by $[i_1\cdots i_r]\mapsto\sum_a z_a^{i_1}\cdots z_a^{i_r}$.
\end{definition}

\begin{proposition}[Forward lift of a pure-jump path equals iterated sums]
\label{prop:ito-sums}
Let $X:[0,T]\to E$ be a finite-variation pure-jump \cadlag{} path, constant between jumps, with jumps $z_a=\Delta X_{\tau_a}$ at $0<\tau_1<\cdots<\tau_n\le T$, and let $\Sig^I(X)$ be its forward lift
\begin{equation}
\Sig^I_{0,t}=\one+\int_{(0,t]}\Sig^I_{0,u^-}\otimes\dd X_u .
\end{equation}
Then for every word $w=(i_1,\ldots,i_k)$,
\begin{equation}
\inner{\Sig^I(X)_{0,T}}{w}
=\sum_{1\le a_1<\cdots<a_k\le n}z_{a_1}^{i_1}\cdots z_{a_k}^{i_k}
=\inner{\Sig^{\mathrm{sum}}(z)}{w},
\qquad
\Sig^{\mathrm{sum}}(z)=\oprod_{a=1}^n(\one+z_a).
\label{eq:ito-sums}
\end{equation}
In particular $\Sig^I(X)$ is a character on the quasi-shuffle Hopf algebra.
\end{proposition}

\begin{proof}
Induct on $k$; both sides are $1$ at $k=0$. Since $\dd X=\sum_a z_a\,\delta_{\tau_a}$, the defining recursion gives
\begin{equation}
\inner{\Sig^I_{0,T}}{(i_1\cdots i_k)}=
\sum_{a=1}^n \inner{\Sig^I_{0,\tau_a^-}}{(i_1\cdots i_{k-1})}\,z_a^{i_k}.
\end{equation}
The integrand is evaluated at the left limit $\tau_a^-$, so by the inductive hypothesis
\begin{equation}
\inner{\Sig^I_{0,\tau_a^-}}{(i_1\cdots i_{k-1})}
=\sum_{a_1<\cdots<a_{k-1}<a}z_{a_1}^{i_1}\cdots z_{a_{k-1}}^{i_{k-1}}.
\end{equation}
The strict inequality $a_{k-1}<a$ is forced by the left limit, since only jumps strictly before $\tau_a$ are seen. Multiplying by $z_a^{i_k}$ and summing over $a$ gives
\begin{equation}
\sum_{a_1<\cdots<a_k}z_{a_1}^{i_1}\cdots z_{a_k}^{i_k},
\end{equation}
which is exactly the coordinate of $\Sig^{\mathrm{sum}}(z)=\oprod_a(\one+z_a)$.
\end{proof}

\begin{proposition}[Hoffman transform is exact on the interpolant]
\label{prop:hoffman-exact}
Let $z_1,\ldots,z_n\in\R^d$ be increments and let $X^{\rm PL}$ be the piecewise-linear path that traverses the increment $z_a$ on the $a$-th segment. Then
\begin{equation}
\Sig(X^{\rm PL})_{0,T}
=\expt(z_1)\otimes\cdots\otimes\expt(z_n)
=\Phi\big(\Sig^{\rm sum}(z_1,\ldots,z_n)\big),
\label{eq:hoffman-exact}
\end{equation}
where $\Phi$ is Hoffman's exponential from the quasi-shuffle character to the shuffle character. Thus the Hoffman step has no discretization error; it is an algebraic identity.
\end{proposition}

\begin{proof}
On a straight segment with increment $z_a$, every iterated integral is a symmetric tensor power, so the segment signature is $\expt(z_a)$. Chen's identity multiplies these segment signatures in chronological order. The coordinate formula is the same block-composition expansion as in \cref{def:hoffman}: repeated indices inside a segment become merged letters, weighted by the factorial denominator coming from integration over the simplex. Hence Hoffman's exponential sends the strict iterated-sums character to the piecewise-linear shuffle character \citep{hoffman2000,diehl2020}.
\end{proof}

\begin{theorem}[It\^o is to Marcus as sums are to integrals]
\label{thm:hopf-square}
With $X$ as in \cref{prop:ito-sums} and $\Sig^M(X)=\Sig(X^{\mathrm{PL}})$ the signature of its chord interpolation, equivalently the Marcus lift,
\begin{equation}
\Sig^M(X)=\Phi\big(\Sig^I(X)\big).
\end{equation}
Equivalently, $\Sig^M$ is a shuffle character and the square
\begin{equation}
\begin{array}{ccc}
(z_1,\ldots,z_n) & \xrightarrow{\ \mathrm{iterated\ sums}\ } & \Sig^I\quad(\mathrm{quasi\mbox{-}shuffle})\\[1.5mm]
\big\downarrow{\scriptstyle\mathrm{chord\ fill}} & & \big\downarrow{\scriptstyle\Phi}\\[1.5mm]
X^{\mathrm{PL}} & \xrightarrow{\ \mathrm{iterated\ integrals}\ } & \Sig^M\quad(\mathrm{shuffle})
\end{array}
\label{eq:hopf-square}
\end{equation}
commutes.
\end{theorem}

\begin{proof}
\emph{(a) Segment.} For a straight segment $\gamma(t)=p+tz$, $t\in[0,1]$, one has $\dd\gamma=z\,\dd t$ and
\begin{equation}
\Sig^k(\gamma)=z^{\otimes k}\!\int_{0<t_1<\cdots<t_k<1}\!\dd t_1\cdots\dd t_k=\frac{z^{\otimes k}}{k!},
\end{equation}
hence $\Sig(\gamma)=\expt(z)$.

\emph{(b) Chen.} The chord path $X^{\mathrm{PL}}$ concatenates constant pieces, whose signature is $\one$, with the chords of increments $z_a$. By Chen's identity and reparametrization invariance,
\begin{equation}
\Sig^M(X)=\oprod_{a=1}^n\expt(z_a),
\end{equation}
and therefore
\begin{equation}
\inner{\Sig^M(X)}{(i_1\cdots i_k)}
=\sum_{a_1\le\cdots\le a_k}\frac{1}{\prod_a \mu_a!}\,z_{a_1}^{i_1}\cdots z_{a_k}^{i_k},
\qquad
\mu_a=\#\{p:a_p=a\}.
\end{equation}

\emph{(c) Identification.} Group a weakly increasing tuple into its maximal constant runs: distinct values $b_1<\cdots<b_j$ with run-lengths $\ell_r$ partition the word positions into consecutive blocks $B_1,\ldots,B_j$ with $|B_r|=\ell_r=\mu_{b_r}$, weight $\prod_r 1/\ell_r!$, and value
\begin{equation}
\prod_r\prod_{p\in B_r}z_{b_r}^{i_p}=\prod_r z_{b_r}^{B_r}.
\end{equation}
Summing over $b_1<\cdots<b_j$ gives $\inner{\Sig^I}{[B_1]\cdots[B_j]}$, and summing over all block compositions of $w$ with weights $\prod_r 1/\ell_r!$ is precisely $\Phi$ applied to $\Sig^I$ by \cref{def:hoffman}. Combining this with \cref{prop:ito-sums} gives $\Sig^M(X)=\Phi(\Sig^I(X))$.
\end{proof}

\begin{remark}[level-two diagonal and the jump correction]
The two compositions of $(i,j)$ are $(i)(j)$ and the merged block $[ij]$, so
\begin{equation}
\inner{\Sig^M}{(i,j)}=\sum_{a<b}z_a^i z_b^j+\frac12\sum_a z_a^i z_a^j,
\end{equation}
where the diagonal $\frac12\sum_a z_a^{\otimes2}$ is exactly the Marcus correction $\frac12\sum_u\Delta X_u\otimes\Delta X_u$ of \cref{eq:marcus-correction}. Level one coincides,
\begin{equation}
\inner{\Sig^I}{(i)}=\inner{\Sig^M}{(i)}=X_T^i-X_0^i;
\end{equation}
the left limit in \cref{prop:ito-sums} is the sole source of strictness, hence of the entire \Ito{}--Marcus gap. For a general finite-variation \cadlag{} path the identity holds verbatim on the pure-jump part, with $\Sig^M$ the signature of the completed graph. For countably many jumps with $\sum_a\norm{z_a}<\infty$, the ordered products converge in $T((E))$ and both identities persist.
\end{remark}

Let $z_1,\ldots,z_n\in\R^d$ be increments, and define
\begin{equation}
E_m(z)=\sum_{k=0}^{m}\frac{z^{\otimes k}}{k!},
\qquad
D_m(z)=1+z.
\end{equation}
Let
\begin{equation}
S_m^{\rm PL}(z)=\prod_{i=1}^n E_m(z_i),
\qquad
S_m^{\rm disc}(z)=\prod_{i=1}^n D_m(z_i),
\end{equation}
where products are tensor products and truncated at level $m$. The first object is the level-$m$ signature of the piecewise-linear path with increments $z_i$; the second is the strict iterated-sums tensor.

\begin{theorem}[quantitative discrete-to-continuous exponential correction]
\label{thm:discrete-bound}
Equip $T^{(m)}(\R^d)$ with the projective tensor norm summed over levels. Let
\begin{equation}
V=\sum_{i=1}^n\norm{z_i},
\qquad
\delta=\max_i\norm{z_i}.
\end{equation}
Then
\begin{equation}
\norm{S_m^{\rm PL}(z)-S_m^{\rm disc}(z)}
\le
\exp(2V)\sum_{i=1}^n\sum_{k=2}^{m}\frac{\norm{z_i}^k}{k!}.
\end{equation}
If $\delta\le 1$, then
\begin{equation}
\norm{S_m^{\rm PL}(z)-S_m^{\rm disc}(z)}
\le
\exp(2V+1)\,\delta V.
\end{equation}
In particular, for increments of a Lipschitz path sampled on a mesh of size $\Delta$, with $\norm{z_i}\le L\Delta$, the error is $O(\Delta)$ at fixed total variation.
\end{theorem}

\begin{proof}
The projective tensor norm is submultiplicative. Hence
\begin{equation}
\norm{E_m(z_i)}\le \sum_{k=0}^{m}\frac{\norm{z_i}^k}{k!}\le e^{\norm{z_i}},
\qquad
\norm{D_m(z_i)}\le 1+\norm{z_i}\le e^{\norm{z_i}}.
\end{equation}
Using the telescoping product identity,
\begin{equation}
\prod_{i=1}^n A_i-\prod_{i=1}^n B_i
=
\sum_{j=1}^n
\left(\prod_{i<j}A_i\right)(A_j-B_j)\left(\prod_{i>j}B_i\right),
\end{equation}
with $A_i=E_m(z_i)$ and $B_i=D_m(z_i)$, gives
\begin{equation}
\norm{S_m^{\rm PL}-S_m^{\rm disc}}
\le
\sum_{j=1}^n
\exp\!\left(\sum_{i<j}\norm{z_i}\right)
\norm{E_m(z_j)-D_m(z_j)}
\exp\!\left(\sum_{i>j}\norm{z_i}\right).
\end{equation}
The exponential factors are bounded by $e^{2V}$, and
\begin{equation}
\norm{E_m(z_j)-D_m(z_j)}
\le
\sum_{k=2}^{m}\frac{\norm{z_j}^k}{k!},
\end{equation}
which proves the first bound. If $\delta\le1$, then
\begin{equation}
\sum_{k=2}^{m}\frac{\norm{z_i}^k}{k!}
\le
\norm{z_i}\delta\sum_{k=2}^{\infty}\frac{1}{k!}
\le e\,\delta\norm{z_i}.
\end{equation}
Summing over $i$ gives the second bound.
\end{proof}

\begin{remark}
The theorem does not replace Hoffman's exact quasi-shuffle isomorphism. Instead it quantifies the error made when the within-step exponential correction is ignored. Hoffman's map removes this error algebraically by adding bracket letters; the estimate above controls the unbracketed approximation.
\end{remark}

\section{The geometricity defect: a second-order principle}
\label{sec:defect}

The shuffle identity, \cref{thm:shuffle}, is the algebraic signature of geometricity. It is an exact second-order identity: any departure from it, whether from a non-geometric stochastic lift or from averaging over a random path, is measured by a bracket or a covariance. This subsumes the \Ito{} shuffle-defect in \cref{prop:ito-defect} and governs the passage to expected signatures.

\begin{theorem}[geometricity defect]
\label{thm:geo-defect}
The following two identities describe the two canonical ways shuffle multiplicativity can fail.
\begin{enumerate}[label=\textup{(\roman*)},leftmargin=2.25em]
\item \textup{(Convention defect.)} Let $X$ be a semimartingale and let $\mathbf X^\circ$ denote its forward \Ito{} lift. For all letters $i,j$ for which the displayed quantities are integrable,
\begin{equation}
\inner{\Sig(\mathbf X^\circ)}{i}\inner{\Sig(\mathbf X^\circ)}{j}
-\inner{\Sig(\mathbf X^\circ)}{i\shuffle j}=[X^i,X^j]_T,
\label{eq:convention-defect}
\end{equation}
the quadratic covariation. Thus the \Ito{} lift agrees with a geometric lift at level two exactly when this bracket term vanishes on the coordinate pair.
\item \textup{(Averaging defect.)} Let $\mathbf X$ be a random weakly geometric lift with square-integrable signature coordinates. For all words $u,v$,
\begin{equation}
\inner{\E\Sig(\mathbf X)}{u}\inner{\E\Sig(\mathbf X)}{v}
-\inner{\E\Sig(\mathbf X)}{u\shuffle v}
=-\Cov\left(\inner{\Sig(\mathbf X)}{u},\inner{\Sig(\mathbf X)}{v}\right).
\label{eq:averaging-defect}
\end{equation}
\end{enumerate}
Both defects are second-order: a bracket for a non-geometric convention, and a covariance for a non-degenerate law. The simultaneous zero-defect regime is the geometric deterministic case.
\end{theorem}

\begin{proof}
For (i), the forward product rule gives
\begin{equation}
X^i_TX^j_T-X^i_0X^j_0=\int_0^T X^i_{u^-}\dd X^j_u+
\int_0^T X^j_{u^-}\dd X^i_u+[X^i,X^j]_T.
\end{equation}
The left side is $\inner{\Sig(\mathbf X^\circ)}{i}\inner{\Sig(\mathbf X^\circ)}{j}$ after translating the initial point to zero, and the two integrals sum to $\inner{\Sig(\mathbf X^\circ)}{ij+ji}=\inner{\Sig(\mathbf X^\circ)}{i\shuffle j}$. For a finite-variation jump path this is exactly \cref{prop:ito-defect}; for a continuous semimartingale it is the usual \Ito{}--Stratonovich correction.

For (ii), each realization is group-like, so
\begin{equation}
\inner{\Sig}{u}\inner{\Sig}{v}=\inner{\Sig}{u\shuffle v}\qquad\text{almost surely}.
\end{equation}
Taking expectations and subtracting $\inner{\E\Sig}{u}\inner{\E\Sig}{v}=\E\inner{\Sig}{u}\,\E\inner{\Sig}{v}$ gives \eqref{eq:averaging-defect}.
\end{proof}

\begin{corollary}[group-likeness detects determinism]
\label{cor:grouplike-det}
$\E\Sig(\mathbf X)$ is group-like if and only if every signature coordinate is almost surely constant. Equivalently, $X$ is almost surely a single tree-like equivalence class. Hence the expected signature is itself a signature only in the degenerate deterministic case; in general it lies in the tensor algebra but outside the group of group-like elements.
\end{corollary}

\begin{proof}
Group-likeness is the vanishing of every defect in \eqref{eq:averaging-defect}. Taking $u=v=w$ gives $\Var(\inner{\Sig}{w})=0$ for every word $w$, so every coordinate is almost surely constant. The word basis is countable and separating, hence $\Sig(\mathbf X)$ is almost surely constant. By \cref{thm:hambly}, the constant signature identifies one reduced path, or one tree-like equivalence class. The converse is immediate.
\end{proof}

\section{Expected signatures, kernels, and the defect mass}
\label{sec:kernel}

For paths truncated at level $N$, write
\begin{equation}
k_N(x,y)=\inner{\Sig^{(N)}(x)}{\Sig^{(N)}(y)}
=\sum_{|w|\le N}\inner{\Sig(x)}{w}\inner{\Sig(y)}{w}.
\end{equation}

\begin{proposition}[kernel decomposition and MMD]
\label{prop:kernel-mmd}
Let $X,X'$ be independent with law $\mu$ and let $Y,Y'$ be independent with law $\nu$, all with level-$N$ square-integrable signatures.
\begin{enumerate}[label=\textup{(\roman*)},leftmargin=2.25em]
\item
\begin{equation}
\E_\mu[k_N(X,X')]=\norm{\E_\mu\Sig^{(N)}}^2,
\end{equation}
and the diagonal excess
\begin{equation}
\E_\mu[k_N(X,X)]-\E_\mu[k_N(X,X')]=
\sum_{|w|\le N}\Var_\mu\inner{\Sig(X)}{w}
\label{eq:defect-mass}
\end{equation}
is exactly the total coordinate defect mass from \cref{thm:geo-defect}\textup{(ii)}.
\item
\begin{equation}
\operatorname{MMD}_N(\mu,\nu)^2=\norm{\E_\mu\Sig^{(N)}-\E_\nu\Sig^{(N)}}^2.
\end{equation}
Thus the level-$N$ signature MMD is the Euclidean distance between truncated expected signatures and separates laws up to their level-$N$ expected-signature coordinates.
\end{enumerate}
\end{proposition}

\begin{proof}
Independence gives
\begin{equation}
\E[k_N(X,X')]=\sum_{|w|\le N}\E\inner{\Sig(X)}{w}\,\E\inner{\Sig(X')}{w}
=\sum_{|w|\le N}\inner{\E\Sig^{(N)}}{w}^2.
\end{equation}
Similarly,
\begin{equation}
\E[k_N(X,X)]=\sum_{|w|\le N}\E\inner{\Sig(X)}{w}^2.
\end{equation}
Subtracting yields the variance sum in \eqref{eq:defect-mass}. Expanding the usual squared MMD expression
\begin{equation}
\E_\mu k_N(X,X')-2\E_{\mu\nu}k_N(X,Y)+\E_\nu k_N(Y,Y')
\end{equation}
bilinearly gives the squared norm of the difference of the two expected-signature vectors.
\end{proof}

\section{The free nilpotent group and why averaging leaves it}
\label{sec:carnot}

Truncated signatures of weakly geometric paths take values in the step-$N$ free nilpotent Lie group
\begin{equation}
G^{(N)}(\R^d)=\exp\big(\lie^{(N)}(\R^d)\big)\subset T^{(N)}(\R^d),
\qquad
\dim \lie^{(N)}=\sum_{k=1}^{N}\ell_d(k),
\end{equation}
with $\ell_d$ the Witt function of \cref{cor:witt}. The map $x\mapsto\Sig^{(N)}(x)$ is the Cartan development of $x$ into $G^{(N)}$, and the $p$-variation metric is comparable to the homogeneous Carnot--Caratheodory norm on $G^{(N)}$ \citep{friz2010multidimensional}.

\begin{figure}[t]
\centering
\includegraphics[width=0.95\linewidth]{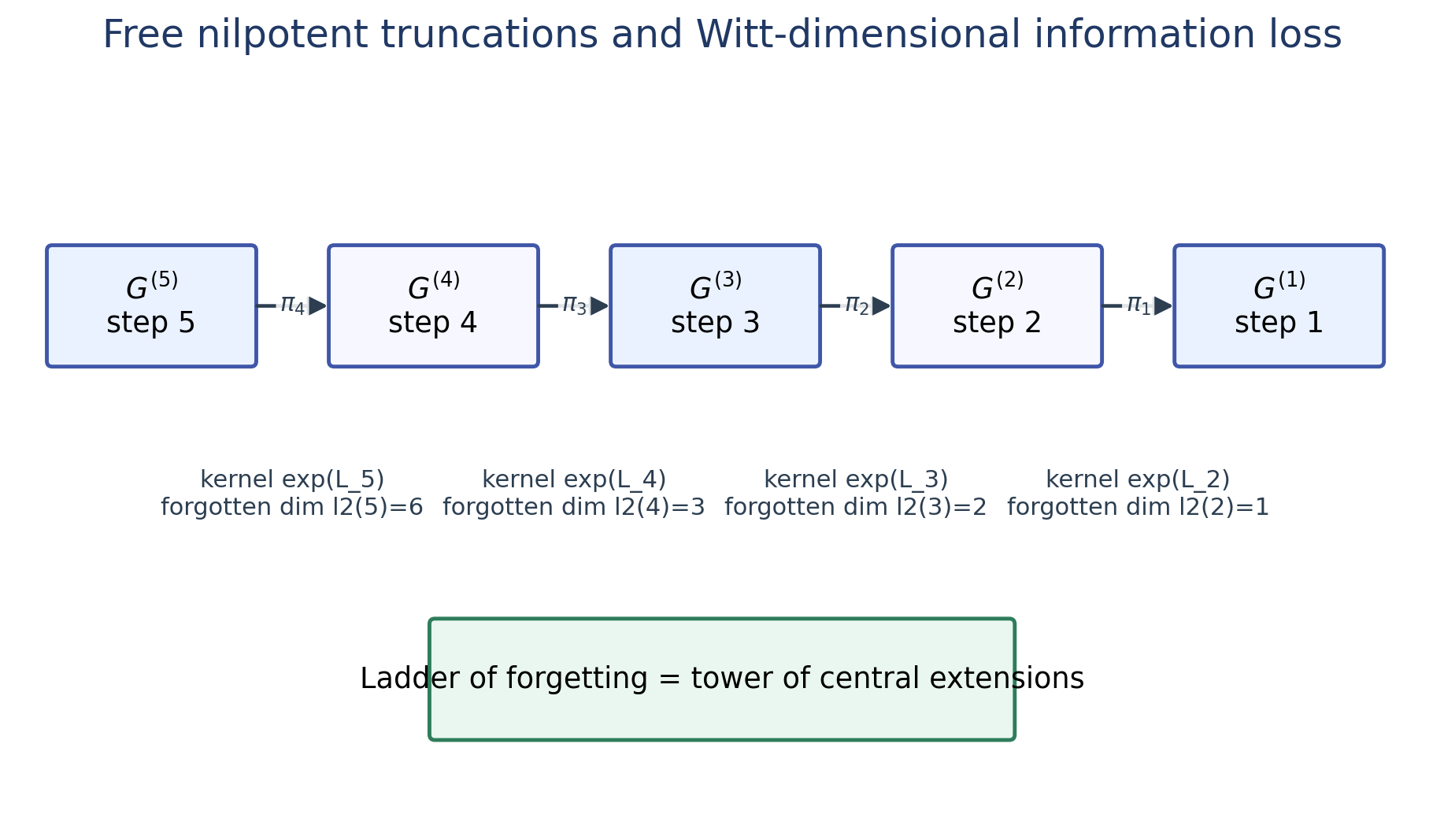}
\caption{The central tower of free nilpotent truncations. Each projection forgets exactly the top-degree free-Lie layer; for $d=2$ the forgotten dimensions through level five are the Witt numbers $2,1,2,3,6$.}
\label{fig:central-tower}
\end{figure}

\begin{proposition}[the ladder of forgetting is a central tower]
\label{prop:central-tower}
The truncations $\pi_N:G^{(N+1)}\to G^{(N)}$ are surjective group homomorphisms with central kernel
\begin{equation}
\ker\pi_N=\exp(\lie_{N+1})\cong\R^{\ell_d(N+1)}.
\end{equation}
Hence the representation ladder
\begin{equation}
\cdots\twoheadrightarrow\Sig^{(N+1)}\twoheadrightarrow\Sig^{(N)}\twoheadrightarrow\cdots
\end{equation}
is a tower of central extensions, each step forgetting exactly the top-degree free-Lie directions, of dimension the Witt number $\ell_d(N+1)$.
\end{proposition}

\begin{proof}
The map $\pi_N$ is the canonical projection $T^{(N+1)}\to T^{(N)}$, an algebra homomorphism that restricts to a group homomorphism on group-like elements. Its kernel on $G^{(N+1)}$ consists of exponentials of Lie elements of homogeneous degree $N+1$, namely $\exp(\lie_{N+1})$. This layer is central in the step-$(N+1)$ algebra because all brackets with positive-degree elements have degree larger than $N+1$ and vanish after truncation. The dimension is $\dim\lie_{N+1}=\ell_d(N+1)$ by \cref{cor:witt}.
\end{proof}

\begin{remark}[averaging leaves the group]
$G^{(N)}$ is a curved, non-convex submanifold of $T^{(N)}$. By \cref{cor:grouplike-det}, the expected signature $\E\Sig^{(N)}$ lies in the ambient algebra but generically not in $G^{(N)}$: it is the barycenter of a distribution on the group, displaced into the enveloping algebra by the covariance mass of \cref{thm:geo-defect}\textup{(ii)}. The probabilistic, algebraic, and geometric statements are one phenomenon: laws live one level up, in the convex hull of the group, and the expected signature is their first moment there.
\end{remark}

\section{Affine and Hawkes expected signatures}
\label{sec:hawkes}

\subsection{Affine expected signatures}
\label{sec:affine-sig}

Let $X$ be a non-explosive affine process on $\R^d$ whose generator is defined on polynomials and has the schematic form
\begin{align}
\mathcal A f(x)
={}&(b_0+b_1x)\cdot\nabla f(x)
+\frac12\Tr\!\big((a_0+a_1x)\nabla^2f(x)\big)\nonumber\\
&+\int_{\R^d}\big(f(x+z)-f(x)-z\cdot\nabla f(x)\big)(m_0+m_1x)(\dd z),
\label{eq:affine-generator}
\end{align}
where the notation indicates affine dependence on the state. Fix a truncation level $m$ and assume that the jump kernel has finite moments through order $m$, that the associated martingale problem is well posed, and that the polynomial moment hierarchy is closed through degree $m$ on the time interval considered. Let $\mathbf X=(t,X_t)$ be the time-augmented Marcus lift and write $\Sig^{(m)}_{0,t}=\Sig^{(m)}(\mathbf X)_{0,t}$.

\begin{proposition}[linear closure of the affine expected signature]
\label{prop:affine-sig}
Fix $m<\infty$. There is a finite vector $U_t$ collecting the expected-signature coordinates $\E\inner{\Sig^{(m)}_{0,t}}{w}$ together with state-weighted coordinates $\E[X_t^{\gamma}\inner{\Sig^{(m)}_{0,t}}{w}]$ for all $|\gamma|+|w|\le m$, and there is a constant matrix $A_m$ depending on the affine characteristics in \eqref{eq:affine-generator} such that
\begin{equation}
\frac{\dd}{\dd t}U_t=A_mU_t,
\qquad
\E[\Sig^{(m)}_{0,t}]=\Pi_m e^{tA_m}U_0,
\label{eq:affine-sig-matrix}
\end{equation}
where $\Pi_m$ selects the coordinates with no state weight.
\end{proposition}

\begin{proof}
The Marcus signature satisfies the last-letter recursion
\begin{equation}
\dd\inner{\Sig_{0,t}}{vi}=\inner{\Sig_{0,t}}{v}\circ\dd \mathbf X_t^i,
\qquad i\in\{0,1,\ldots,d\}.
\end{equation}
Multiplying by a monomial $X_t^\gamma$ and applying Dynkin's formula to the signature-augmented Markov state gives a linear combination of terms of the same form. Affineness of the drift, covariance, and jump compensator implies that the polynomial moment hierarchy closes by total degree; appending a signature letter lowers the remaining word degree. Therefore no coordinate with $|\gamma|+|w|>m$ is needed. The resulting finite closed linear system has the matrix-exponential solution \eqref{eq:affine-sig-matrix}. This is the polynomial-process closure specialized to the signature-augmented state \citep{cuchiero2023}.
\end{proof}

\subsection{Scalar exponential Hawkes closure}

Consider the scalar exponential-kernel Hawkes process \citep{hawkes1971,hawkesOakes1974}
\begin{equation}
\lambda_t=\mu+\int_{(0,t)}\alpha e^{-\beta(t-s)}\dd N_s,
\qquad 0<\alpha<\beta,
\label{eq:hawkes-intensity}
\end{equation}
with initial intensity $\lambda_0=\ell>0$. Equivalently,
\begin{equation}
\dd\lambda_t=\beta(\mu-\lambda_t)\dd t+\alpha\dd N_t.
\label{eq:lambda-sde}
\end{equation}
Let
\begin{equation}
X_t=(t,N_t)\in\R^2,
\end{equation}
with alphabet $\{0,1\}$, where $0$ denotes time and $1$ denotes the count coordinate. The Marcus signature of $X$ is the ordinary signature of the completed graph obtained by adding a vertical unit chord at each jump.

\begin{figure}[t]
\centering
\includegraphics[width=0.95\linewidth]{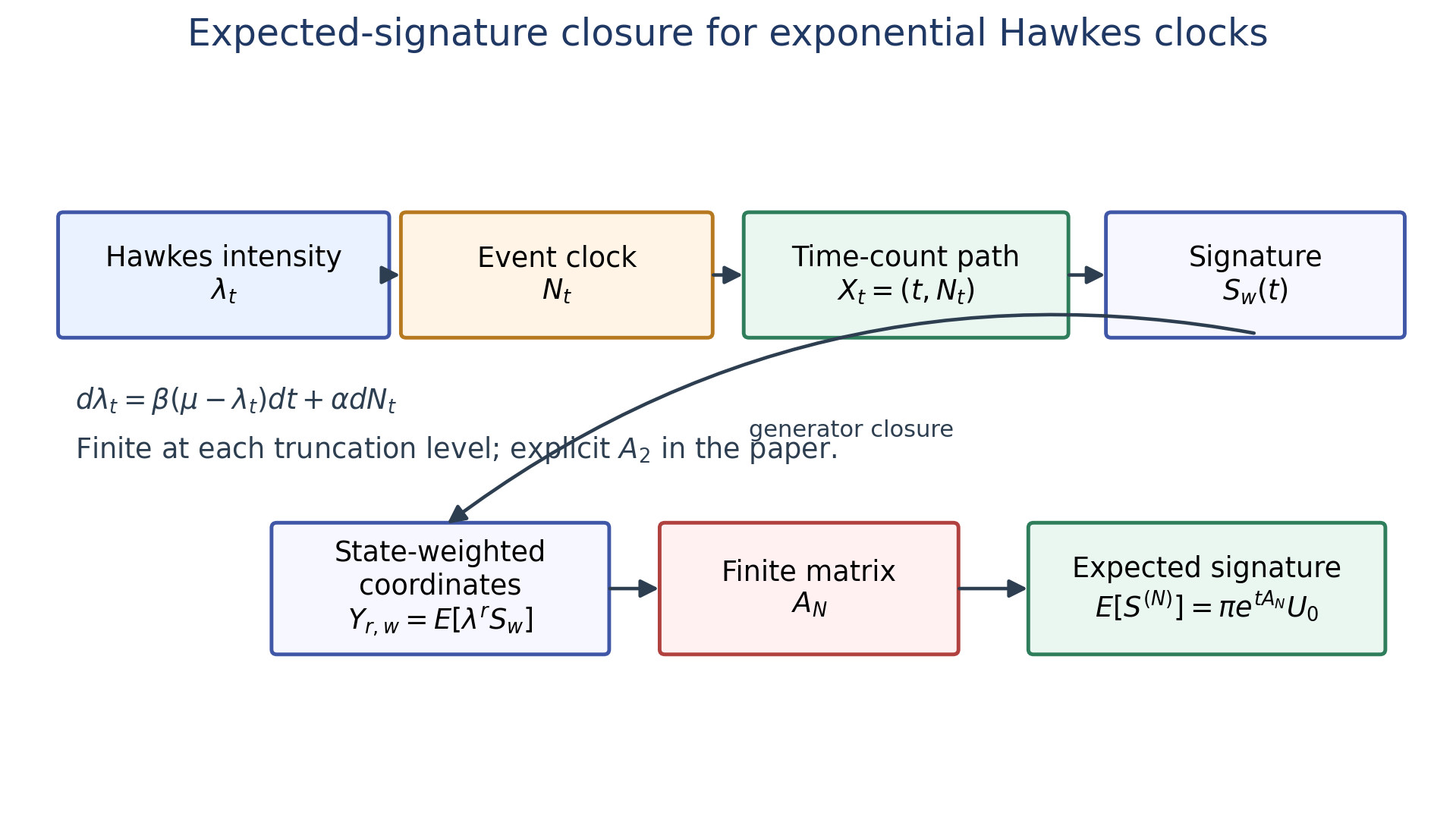}
\caption{Finite expected-signature closure for the exponential Hawkes clock. State-weighted signature coordinates evolve under a finite matrix $A_N$, giving $\E[\Sig^{(N)}]=\pi e^{tA_N}U_0$ at fixed truncation.}
\label{fig:hawkes-closure}
\end{figure}

For a word $w\in\{0,1\}^{\ast}$, let $S_w(t)=\inner{\Sig^M(X)_{0,t}}{w}$. Set $S_{\emptyword}(t)=1$. If $w$ ends in $0$, write $w=v0$ and set $\partial_0 w=v$; otherwise $\partial_0 w=\bot$. Let $\rho(w)$ be the number of trailing $1$'s in $w$. For $0\le k\le \rho(w)$, let $w^{[-k]}$ be the word obtained by deleting the last $k$ letters from $w$, with $w^{[0]}=w$.

\begin{lemma}[signature dynamics for a time-count Marcus path]
\label{lem:signature-dynamics}
Between jumps,
\begin{equation}
\frac{\dd}{\dd t}S_w(t)=
\begin{cases}
S_v(t),& w=v0,\\
0,& \text{otherwise.}
\end{cases}
\end{equation}
At a jump of $N$, the signature is right-multiplied by $\exp_{\otimes}(e_1)$, so
\begin{equation}
S_w(t)=\sum_{k=0}^{\rho(w)}\frac{1}{k!}S_{w^{[-k]}}(t-).
\label{eq:jump-Sw}
\end{equation}
\end{lemma}

\begin{proof}
Between jumps the path moves only in the time direction by increment $e_0\dd t$. Right multiplication by $\exp_{\otimes}(e_0\dd t)=1+e_0\dd t+O(\dd t^2)$ gives a first-order contribution only to words ending in $0$, namely from their prefix $v$.

At a jump, the Marcus completed graph traverses a vertical segment of increment $e_1$. By Chen's identity, the new signature is the old signature tensor-multiplied by $\exp_{\otimes}(e_1)$. The coordinate of $w$ in this product is obtained by taking $k$ terminal letters from the jump exponential. This is possible exactly when the last $k$ letters of $w$ are all $1$, and the exponential coefficient is $1/k!$. Summing over $k=0,\ldots,\rho(w)$ gives \eqref{eq:jump-Sw}.
\end{proof}

\begin{theorem}[finite-dimensional Hawkes expected-signature closure]
\label{thm:hawkes-closure}
Fix a truncation level $m\in\N$. For all pairs $(r,w)$ satisfying
\begin{equation}
0\le r\le m,
\qquad
w\in\{0,1\}^{\ast},
\qquad
r+|w|\le m,
\end{equation}
define
\begin{equation}
Y_{r,w}(t)=\E\big[\lambda_t^r S_w(t)\big].
\end{equation}
Then the vector $Y^{(m)}(t)=(Y_{r,w}(t))_{r+|w|\le m}$ satisfies a finite-dimensional linear ordinary differential equation
\begin{equation}
\frac{\dd}{\dd t}Y^{(m)}(t)=A_mY^{(m)}(t),
\qquad
Y_{r,\emptyword}(0)=\ell^r,
\qquad
Y_{r,w}(0)=0\quad(w\ne\emptyword),
\label{eq:hawkes-matrix}
\end{equation}
where the matrix entries are determined by the coordinate formula
\begin{align}
\frac{\dd}{\dd t}Y_{r,w}
={}&
\one_{\{w=v0\}}Y_{r,v}
+r\beta\mu\,Y_{r-1,w}-r\beta\,Y_{r,w}
\nonumber\\
&+
\sum_{j=0}^{r-1}\binom{r}{j}\alpha^{r-j}Y_{j+1,w}
+
\sum_{j=0}^{r}\binom{r}{j}\alpha^{r-j}
\sum_{k=1}^{\rho(w)}\frac{1}{k!}Y_{j+1,w^{[-k]}}.
\label{eq:closure-formula}
\end{align}
Terms with negative intensity powers are omitted, so the drift term $r\beta\mu Y_{r-1,w}$ is absent when $r=0$. Consequently,
\begin{equation}
\E[\Sig^{M,(m)}(X)_{0,T}]
=
\big(Y_{0,w}(T)\big)_{|w|\le m}
=
\Pi_m\exp(TA_m)Y^{(m)}(0),
\end{equation}
where $\Pi_m$ selects the coordinates with $r=0$.
\end{theorem}

\begin{proof}
The process $(\lambda_t,S_w(t):|w|\le m)$ is a piecewise-deterministic Markov process \citep{davis1984}. Apply the generator to the function
\begin{equation}
f_{r,w}(\lambda,S)=\lambda^rS_w.
\end{equation}
The continuous part consists of the intensity drift $\dot\lambda=\beta(\mu-\lambda)$ and the horizontal signature drift from \cref{lem:signature-dynamics}. Thus it contributes
\begin{equation}
r\lambda^{r-1}\beta(\mu-\lambda)S_w
+
\one_{\{w=v0\}}\lambda^rS_v.
\end{equation}
Taking expectations gives
\begin{equation}
r\beta\mu Y_{r-1,w}-r\beta Y_{r,w}+\one_{\{w=v0\}}Y_{r,v}.
\end{equation}

At a jump, which occurs with stochastic rate $\lambda$, the intensity changes from $\lambda$ to $\lambda+\alpha$ and the signature coordinate changes according to \eqref{eq:jump-Sw}. The jump contribution to the generator is therefore
\begin{equation}
\lambda\left((\lambda+\alpha)^r\sum_{k=0}^{\rho(w)}\frac{1}{k!}S_{w^{[-k]}}-\lambda^rS_w\right).
\end{equation}
Expand $(\lambda+\alpha)^r=\sum_{j=0}^r\binom{r}{j}\alpha^{r-j}\lambda^j$. The term with $k=0$ and $j=r$ is $\lambda^{r+1}S_w$, which cancels the subtracted old-state term. The remaining $k=0$ terms have $j\le r-1$ and yield
\begin{equation}
\sum_{j=0}^{r-1}\binom{r}{j}\alpha^{r-j}Y_{j+1,w}.
\end{equation}
The terms with $k\ge1$ yield
\begin{equation}
\sum_{j=0}^{r}\binom{r}{j}\alpha^{r-j}
\sum_{k=1}^{\rho(w)}\frac{1}{k!}Y_{j+1,w^{[-k]}}.
\end{equation}
This proves \eqref{eq:closure-formula}.

It remains to check closure. In the $k=0$ jump terms, $j+1+|w|\le r+|w|\le m$ because $j\le r-1$. In the $k\ge1$ terms,
\begin{equation}
j+1+|w^{[-k]}|=j+1+|w|-k\le r+|w|\le m.
\end{equation}
The drift and time terms also remain within total degree $m$. Hence all right-hand-side coordinates belong to the finite vector $Y^{(m)}$. This gives the finite linear ODE and the matrix-exponential solution.
\end{proof}

\begin{theorem}[explicit level-two Hawkes matrix]
\label{thm:hawkes-A2}
At truncation level two, set
\begin{equation}
U_t=\big(Y_{0,\emptyword},Y_{0,0},Y_{0,1},Y_{0,00},Y_{0,01},Y_{0,10},Y_{0,11},Y_{1,\emptyword},Y_{1,0},Y_{1,1},Y_{2,\emptyword}\big)^{\top}.
\end{equation}
Then $\dot U_t=A_2U_t$, with
\begin{equation}
\resizebox{\textwidth}{!}{$
A_2=
\begin{pmatrix}
0&0&0&0&0&0&0&0&0&0&0\\
1&0&0&0&0&0&0&0&0&0&0\\
0&0&0&0&0&0&0&1&0&0&0\\
0&1&0&0&0&0&0&0&0&0&0\\
0&0&0&0&0&0&0&0&1&0&0\\
0&0&1&0&0&0&0&0&0&0&0\\
0&0&0&0&0&0&0&\tfrac12&0&1&0\\
\beta\mu&0&0&0&0&0&0&\alpha-\beta&0&0&0\\
0&\beta\mu&0&0&0&0&0&1&\alpha-\beta&0&0\\
0&0&\beta\mu&0&0&0&0&\alpha&0&\alpha-\beta&1\\
0&0&0&0&0&0&0&2\beta\mu+\alpha^2&0&0&2(\alpha-\beta)
\end{pmatrix}.
$}
\label{eq:A2}
\end{equation}
With $U_0=(1,0,0,0,0,0,0,\ell,0,0,\ell^2)^\top$, the level-two expected-signature vector is obtained by selecting the first seven coordinates of $e^{TA_2}U_0$.
\end{theorem}

\begin{proof}
This is \cref{thm:hawkes-closure} specialized to all pairs $r+|w|\le2$. For example, $\dot Y_{0,1}=Y_{1,\emptyword}$, $\dot Y_{0,01}=Y_{1,0}$, $\dot Y_{0,11}=Y_{1,1}+\tfrac12Y_{1,\emptyword}$, and $\dot Y_{2,\emptyword}=(2\beta\mu+\alpha^2)Y_{1,\emptyword}+2(\alpha-\beta)Y_{2,\emptyword}$. Listing these equations in the displayed order gives \eqref{eq:A2}.
\end{proof}

\begin{corollary}[closed level-one and level-two coordinates]
\label{cor:level-two}
Let $c=\beta-\alpha>0$ and
\begin{equation}
\lambda_\infty=\frac{\beta\mu}{c}.
\end{equation}
Then
\begin{align}
\E[N_T]
&=\lambda_\infty T+(\ell-\lambda_\infty)\frac{1-e^{-cT}}{c},
\label{eq:EN}\\
\E\left[\int_0^T t\dd N_t\right]
&=\frac{\lambda_\infty T^2}{2}
+(\ell-\lambda_\infty)
\frac{1-(1+cT)e^{-cT}}{c^2}\label{eq:EtN}
\end{align}
Moreover,
\begin{equation}
\E\left[S_{01}(T)\right]=\E\left[\int_0^T t\dd N_t\right],
\qquad
\E\left[S_{10}(T)\right]=T\E[N_T]-\E\left[S_{01}(T)\right],
\end{equation}
and
\begin{equation}
\E[S^{M}_{11}(T)]-\E[S^{I}_{11}(T)]=\frac12\E[N_T].
\end{equation}
\end{corollary}

\begin{proof}
Taking expectations in \eqref{eq:lambda-sde} gives
\begin{equation}
\frac{\dd}{\dd t}\E[\lambda_t]
=\beta\mu-(\beta-\alpha)\E[\lambda_t],
\qquad
\E[\lambda_t]=\lambda_\infty+(\ell-\lambda_\infty)e^{-ct}.
\end{equation}
Since $N_t-\int_0^t\lambda_s\dd s$ is a martingale,
\begin{equation}
\E[N_T]=\int_0^T\E[\lambda_t]\dd t,
\qquad
\E\left[\int_0^T t\dd N_t\right]=\int_0^T t\E[\lambda_t]\dd t.
\end{equation}
Evaluating these elementary integrals gives \eqref{eq:EN} and \eqref{eq:EtN}. The identities for $S_{01}$ and $S_{10}$ follow from the definitions of the time-count signature and the level-two shuffle relation. The final formula is the expectation of \cref{prop:jump-correction}, because unit jumps satisfy $\sum(\Delta N)^2=N_T$.
\end{proof}

\begin{theorem}[local identifiability from the first expected-signature coordinate]
\label{thm:hawkes-identifiability}
Assume $\lambda_0=\ell=\mu$ and $\alpha>0$. Let
\begin{equation}
F(T)=\E[N_T].
\end{equation}
Then the Hawkes parameters are recovered from $F$ near $T=0$ by
\begin{equation}
\mu=F'(0),
\qquad
\alpha=\frac{F''(0)}{F'(0)},
\qquad
\beta=\frac{F''(0)}{F'(0)}-\frac{F'''(0)}{F''(0)}.
\label{eq:ident-reconstruct}
\end{equation}
Thus the first expected-signature coordinate $T\mapsto\inner{\E[\Sig^M(X)_{0,T}]}{1}$ locally identifies $(\mu,\alpha,\beta)$.
\end{theorem}

\begin{proof}
By \cref{cor:level-two}, $F'(T)=\E[\lambda_T]$. With $\ell=\mu$,
\begin{equation}
F'(T)=\lambda_\infty+(\mu-\lambda_\infty)e^{-cT},
\qquad c=\beta-\alpha,
\qquad \lambda_\infty=\frac{\beta\mu}{c}.
\end{equation}
Therefore
\begin{equation}
F'(0)=\mu.
\end{equation}
Also
\begin{equation}
\mu-\lambda_\infty
=\mu-\frac{\beta\mu}{\beta-\alpha}
=-\frac{\mu\alpha}{\beta-\alpha}
=-\frac{\mu\alpha}{c}.
\end{equation}
Hence
\begin{equation}
F''(0)=-c(\mu-\lambda_\infty)=\mu\alpha.
\end{equation}
Since $\alpha>0$ and $\mu>0$, $F''(0)>0$ and
\begin{equation}
\alpha=\frac{F''(0)}{F'(0)}.
\end{equation}
A further derivative gives
\begin{equation}
F'''(0)=c^2(\mu-\lambda_\infty)=-\mu\alpha c.
\end{equation}
Thus
\begin{equation}
c=-\frac{F'''(0)}{F''(0)},
\qquad
\beta=\alpha+c
=\frac{F''(0)}{F'(0)}-\frac{F'''(0)}{F''(0)}.
\end{equation}
This proves \eqref{eq:ident-reconstruct}.
\end{proof}

\begin{theorem}[collapsed count signature and affine transform]
\label{thm:affine-pgf}
For the unaugmented scalar counting path $N$, the Marcus signature collapses to
\begin{equation}
\Sig^M(N)_{0,T}=\exp_{\otimes}(N_T e_1).
\end{equation}
Consequently, its expected signature is determined by the moment generating function
\begin{equation}
M(T,\theta)=\E[e^{\theta N_T}].
\end{equation}
For the exponential Hawkes model, $M$ has the affine representation
\begin{equation}
M(T,\theta)=\exp\{A(T,\theta)+B(T,\theta)\ell\},
\end{equation}
where
\begin{equation}
\frac{\partial B}{\partial T}= -\beta B+e^{\theta+\alpha B}-1,
\qquad
\frac{\partial A}{\partial T}=\beta\mu B,
\qquad
A(0,\theta)=B(0,\theta)=0.
\end{equation}
\end{theorem}

\begin{proof}
In one dimension, every bounded-variation signature equals the tensor exponential of the total increment; this follows from the shuffle identity or from the symmetry of one-dimensional iterated integrals. Thus $\Sig^M(N)_{0,T}=\exp_{\otimes}(N_Te_1)$.

For the affine transform, define
\begin{equation}
u(T,\lambda)=\E_{\lambda}[e^{\theta N_T}].
\end{equation}
The Markov generator acting on functions of $(n,\lambda)$ is
\begin{equation}
\G f(n,\lambda)=\beta(\mu-\lambda)\partial_\lambda f(n,\lambda)
+\lambda\{f(n+1,\lambda+\alpha)-f(n,\lambda)\}.
\end{equation}
Use the ansatz $f(n,\lambda)=e^{\theta n}e^{A(T)+B(T)\lambda}$. The jump part contributes
\begin{equation}
\lambda\left(e^{\theta}e^{\alpha B}-1\right)f,
\end{equation}
and the drift part contributes
\begin{equation}
\beta(\mu-\lambda)Bf.
\end{equation}
Equating coefficients of $1$ and $\lambda$ in the backward equation gives the displayed ODEs for $A$ and $B$. The initial condition is $M(0,\theta)=1$.
\end{proof}

\begin{remark}
\Cref{thm:affine-pgf} explains why time augmentation is essential for path reconstruction. The unaugmented scalar signature contains all moments of the terminal count $N_T$, but it cannot record the event times. The time-augmented closure theorem, \cref{thm:hawkes-closure}, is the corresponding path-level object.
\end{remark}

\begin{remark}[no Riccati claim for log-signatures]
The affine transform in \cref{thm:affine-pgf} is only the scalar terminal-count transform of the collapsed one-dimensional path. It is not a Riccati formula for log-signature coordinates. The genuinely path-level expected-signature object in this paper is the finite linear moment closure of \cref{prop:affine-sig,thm:hawkes-closure}.
\end{remark}

\section{Directional cross-area for multivariate Hawkes processes}
\label{sec:cross-area}

For a two-channel counting path $N=(N^1,N^2)$, define the antisymmetric second-level area
\begin{equation}
A_T^{12}=\frac12\left(\inner{\Sig^M}{12}-\inner{\Sig^M}{21}\right)
=\frac12\int_0^T\big(N^1_{s-}\dd N^2_s-N^2_{s-}\dd N^1_s\big).
\label{eq:cross-area}
\end{equation}
The symmetric second level is constrained by shuffle identities; the antisymmetric part is the signature-native coordinate that records ordering.

\begin{proposition}[cross-area as an affine moment]
\label{prop:cross-area-affine}
For a multivariate exponential Hawkes process with common decay, $\E[A_T^{12}]$ is a coordinate of the affine expected-signature closure in \cref{prop:affine-sig}. Equivalently,
\begin{equation}
\E[A_T^{12}]
=\frac12\int_0^T\left(\E[N^1_s\lambda^2_s]-\E[N^2_s\lambda^1_s]\right)\dd s,
\label{eq:cross-area-intensity}
\end{equation}
and the vector of moments $\E[N^i_s\lambda^j_s]$ closes linearly after adjoining the Hawkes memory variables.
\end{proposition}

\begin{proof}
The compensator identity gives $\E[\int H_s\dd N^j_s]=\int\E[H_s\lambda^j_s]\dd s$ for bounded predictable $H$. Applying this with $H_s=N^i_{s-}$ gives \eqref{eq:cross-area-intensity}. For exponential kernels, $(N,\lambda)$ is Markov affine; the products $N^i\lambda^j$ are degree-two affine-polynomial moments, so the same Dynkin closure as in \cref{prop:affine-sig} applies.
\end{proof}

\begin{theorem}[leading-order directional excitation law]
\label{thm:cross-area-leading}
Consider a two-channel Hawkes process with equal baselines $\mu$, common decay $\beta$, zero self-excitation, and cross-kernels
\begin{equation}
\phi_{12}(t)=a_{12}e^{-\beta t},\qquad \phi_{21}(t)=a_{21}e^{-\beta t},
\end{equation}
where $a_{ij}$ is the excitation of channel $i$ by channel $j$. As $(a_{12},a_{21})\to0$,
\begin{equation}
\E[A_T^{12}]=c_T(a_{21}-a_{12})+O(\|(a_{12},a_{21})\|^2),
\label{eq:cross-leading}
\end{equation}
where
\begin{equation}
 c_T=\frac12\int_0^T\left\{\mu K_s+\mu^2\left(sK_s-H_s\right)\right\}\dd s,
\qquad
K_s=\int_0^s e^{-\beta u}\dd u,
\qquad
H_s=\int_0^s K_u\dd u.
\label{eq:cross-coeff}
\end{equation}
Moreover $c_T>0$ for every $T>0$.
\end{theorem}

\begin{proof}
By \eqref{eq:cross-area-intensity}, it suffices to expand $\E[N^1_s\lambda^2_s]-\E[N^2_s\lambda^1_s]$ to first order in $a_{12},a_{21}$. At zero excitation, $N^1$ and $N^2$ are independent Poisson processes with rate $\mu$. The mean count expansion is
\begin{equation}
\E[N^i_s]=\mu s+\mu a_{ij}H_s+O(a^2),\qquad i\ne j.
\end{equation}
Also,
\begin{equation}
\lambda^2_s=\mu+a_{21}\int_0^s e^{-\beta(s-u)}\dd N^1_u+O(a^2),
\qquad
\lambda^1_s=\mu+a_{12}\int_0^s e^{-\beta(s-u)}\dd N^2_u+O(a^2).
\end{equation}
For a Poisson process $P$ of rate $\mu$,
\begin{equation}
\E\left[P_s\int_0^s e^{-\beta(s-u)}\dd P_u\right]=\mu^2sK_s+\mu K_s,
\end{equation}
where the second term is the diagonal contribution of the same jump. Therefore
\begin{align}
\E[N^1_s\lambda^2_s]-\E[N^2_s\lambda^1_s]
&=(a_{21}-a_{12})\left\{\mu K_s+\mu^2(sK_s-H_s)\right\}+O(a^2).
\end{align}
Integrating over $s$ and multiplying by $1/2$ proves \eqref{eq:cross-leading}. Since $K_s$ is increasing and nonzero, $\mu K_s>0$ for $s>0$ and $sK_s-H_s=\int_0^s(K_s-K_u)\dd u\ge0$, hence $c_T>0$.
\end{proof}

\section{Reversal, antipodes, and cross-area sign}
\label{sec:antipode}

\begin{theorem}[reversal is the antipode]
\label{thm:antipode}
For a weakly geometric path $X$ with reversal $\overleftarrow X$ and every word $w=(i_1\cdots i_k)$,
\begin{equation}
\inner{\Sig(\overleftarrow X)}{i_1\cdots i_k}=(-1)^k\,\inner{\Sig(X)}{i_k\cdots i_1}.
\end{equation}
\end{theorem}

\begin{proof}
By Chen's identity, $\Sig(\overleftarrow X)_{0,T}=\Sig(X)_{0,T}^{-1}$. The inverse of a group-like series is the antipode $S$ of the tensor Hopf algebra. On words the antipode is
\begin{equation}
S(i_1\cdots i_k)=(-1)^k(i_k\cdots i_1),
\end{equation}
and coordinate duality gives $\inner{g^{-1}}{w}=\inner{g}{S(w)}$.
\end{proof}

\begin{corollary}[the cross-area is reversal-odd]
\label{cor:area-odd}
For
\begin{equation}
A^{ij}(X)=\frac12\left(\inner{\Sig(X)}{ij}-\inner{\Sig(X)}{ji}\right),
\end{equation}
one has
\begin{equation}
A^{ij}(\overleftarrow X)=-A^{ij}(X).
\end{equation}
Hence for a process whose time reversal exchanges the excitation roles, the expected cross-area changes sign, as in the reversal experiment of \cref{sec:numerics}.
\end{corollary}

\begin{proof}
By \cref{thm:antipode} at level two,
\begin{equation}
\inner{\Sig(\overleftarrow X)}{ij}=\inner{\Sig(X)}{ji},
\qquad
\inner{\Sig(\overleftarrow X)}{ji}=\inner{\Sig(X)}{ij}.
\end{equation}
Subtracting gives the claim.
\end{proof}

\section{Moment threshold for heavy-tailed drivers}
\label{sec:stable-threshold}

\begin{theorem}[stable-law moment threshold for the expected signature]
\label{thm:stable-threshold}
Let $X$ be a one-dimensional symmetric $\gamma$-stable Levy process with $\gamma\in(0,2)$ and let $\Sig^M(X)$ be its geometric Marcus lift. For every integer $k\ge1$,
\begin{equation}
\E\left[\left|\inner{\Sig^M(X)_{0,T}}{1^k}\right|\right]<\infty
\quad\Longleftrightarrow\quad
k<\gamma.
\end{equation}
Consequently the full expected signature is not finite beyond the levels allowed by the stable moment threshold, and expected-signature characterization is unavailable at every level $k\ge\lceil\gamma\rceil$.
\end{theorem}

\begin{proof}
In one dimension the geometric signature is the tensor exponential of the total increment:
\begin{equation}
\inner{\Sig^M(X)_{0,T}}{1^k}=\frac{X_T^k}{k!}.
\end{equation}
A symmetric $\gamma$-stable random variable has finite absolute moment of order $k$ if and only if $k<\gamma$ \citep{sato1999}. The equivalence follows immediately. The final statement is the contrapositive: a finite expected signature through level $k$ would require the displayed coordinate to be integrable.
\end{proof}

\begin{proposition}[normalized expected signature is characteristic without moments]
\label{prop:normalized}
Under the characteristic-kernel hypotheses of \citet{chevyrev2022}, there is a tensor normalization $\Lambda$ for which $\Lambda(\Sig(X))$ is uniformly bounded on the relevant path space, so $\E[\Lambda(\Sig(X))]$ exists for every probability law. The induced maximum mean discrepancy is a metric on laws. In particular, for heavy-tailed drivers excluded by the stable threshold in \cref{thm:stable-threshold}, where raw coordinates $\E[\Sig^k]$ may diverge, the law can still be characterized by the normalized expected signature.
\end{proposition}

\begin{proof}
This is the normalized-signature-kernel construction of \citet{chevyrev2022}. The normalization makes the feature map bounded, so all first moments of the normalized tensor feature exist. Characteristicness of the associated kernel gives injectivity of the mean embedding, and the corresponding MMD is a metric on laws. The stable-threshold theorem only obstructs raw polynomial moments; bounded normalized features are not subject to that moment threshold.
\end{proof}

\begin{theorem}[signature continuity]
\label{thm:itolyons}
Fix $N<\infty$ and a bounded set of weakly geometric $p$-rough paths. The map
\begin{equation}
\mathbf x\longmapsto \Sig^{(N)}(\mathbf x)
\end{equation}
from the $p$-variation rough-path topology to $T^{(N)}(\R^d)$ is continuous.
\end{theorem}

\begin{proof}
For levels at most $\lfloor p\rfloor$ this is part of the rough-path topology. Higher levels up to $N$ are obtained by Lyons' extension theorem, and the extension map is continuous in $p$-variation on bounded sets \citep{lyons1998,lcl2007,frizhairer2020}.
\end{proof}

\begin{corollary}[signature large deviations]
\label{cor:ldp}
If a family of driving rough paths satisfies a large-deviation principle in $p$-variation with good rate function $I$, then their level-$N$ signatures satisfy a large-deviation principle with good rate function
\begin{equation}
I_{\Sig}(g)=\inf\{I(\mathbf x):\Sig^{(N)}(\mathbf x)=g\}.
\end{equation}
\end{corollary}

\begin{proof}
Apply the contraction principle to the continuous map in \cref{thm:itolyons}.
\end{proof}

\section{Numerical validation}
\label{sec:numerics}

The reproducibility script included with the paper verifies five groups of identities.

\begin{enumerate}[label=(\roman*),leftmargin=2.25em]
\item Chen multiplicativity, the shuffle identity, tree-like cancellation, and Witt dimensions for signatures of piecewise-linear paths in $\R^2$.
\item The pure-jump Hopf square: forward \Ito{} equals iterated sums, and Marcus equals the Hoffman/geometric exponential.
\item The Marcus--\Ito{} jump correction and the \Ito{} shuffle-defect formula for simulated jump paths.
\item The Hawkes formulas \eqref{eq:EN}, \eqref{eq:EtN}, the explicit level-two matrix \eqref{eq:A2}, and the parameter reconstruction \eqref{eq:ident-reconstruct} using exact Ogata thinning.
\item The cross-area sign law for a two-channel Hawkes process and its reversal.
\end{enumerate}

For the Hawkes experiment, the parameters are
\begin{equation}
\mu=0.5,
\qquad
\alpha=0.8,
\qquad
\beta=1.0,
\qquad
T=10,
\qquad
\lambda_0=\mu.
\end{equation}

\begin{figure}[t]
\centering
\includegraphics[width=0.88\linewidth]{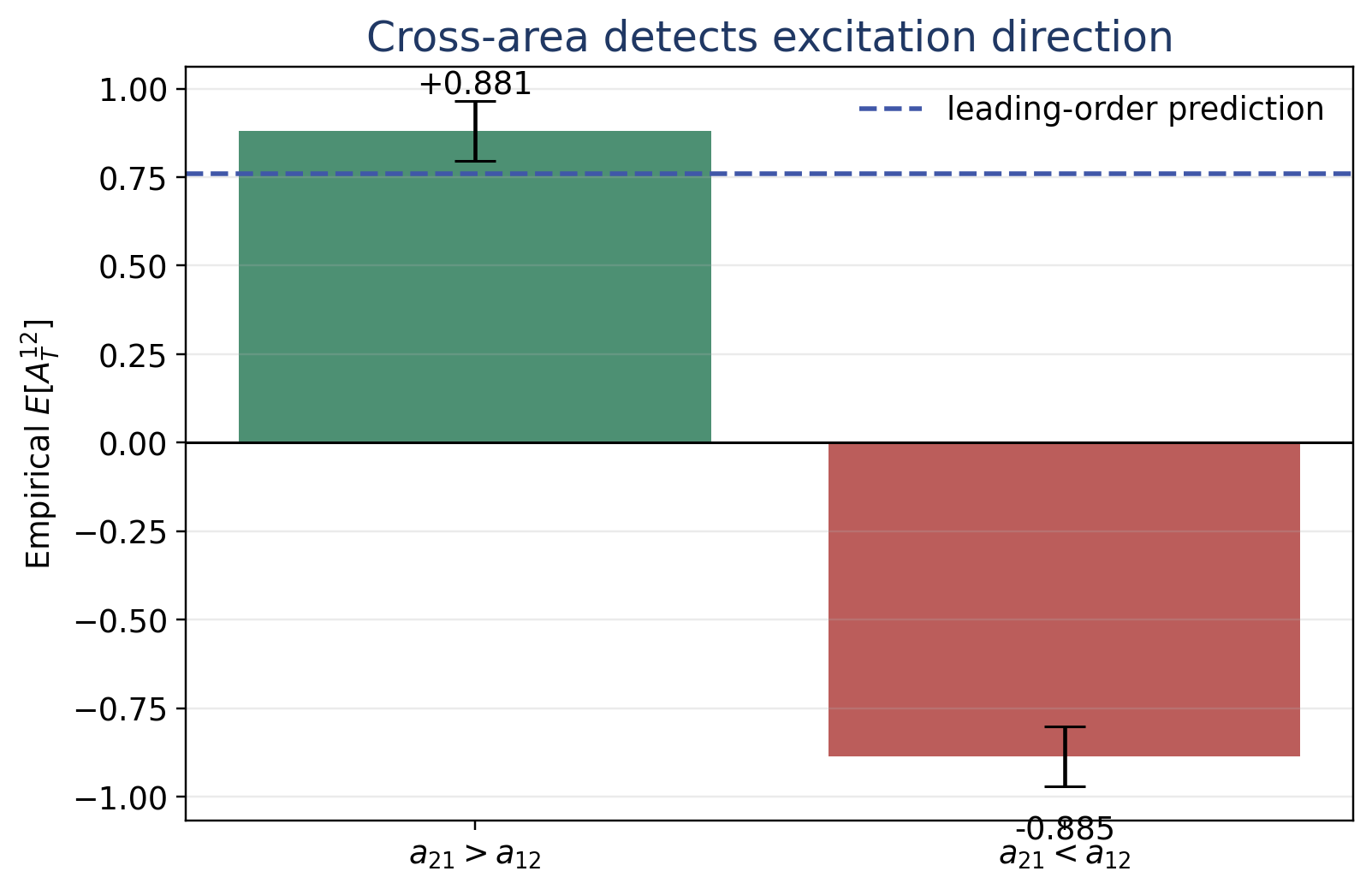}
\caption{Directional cross-area validation. Reversing the excitation asymmetry reverses the sign of the antisymmetric level-two coordinate, matching the leading-order directionality theorem.}
\label{fig:cross-area-numerics}
\end{figure}

The branching ratio is $\alpha/\beta=0.8$.

\begin{table}[t]
\centering
\renewcommand{\arraystretch}{1.18}
\begin{tabular}{@{}lrrr@{}}
\toprule
Quantity & Monte Carlo & Closed form & Error / SE \\
\midrule
$\E[N_T]$ & 16.3232 & 16.3534 & 0.75 \\
$\E[\int_0^T t\dd N_t]$ & 95.0867 & 95.3003 & 0.85 \\
Marcus--\Ito{} gap max error & \multicolumn{3}{r}{0.000} \\
Recovered $\mu$ & 0.5000 & 0.5000 & -- \\
Recovered $\alpha$ & 0.8000 & 0.8000 & -- \\
Recovered $\beta$ & 1.0000 & 1.0000 & -- \\
\bottomrule
\end{tabular}
\caption{Numerical validation of the Hawkes expected-signature identities and local identifiability formulas.}
\label{tab:hawkes-num}

\end{table}

\begin{table}[t]
\centering
\renewcommand{\arraystretch}{1.18}
\begin{tabular}{@{}lrr@{}}
\toprule
Cross-area experiment & Monte Carlo mean & Standard error \\
\midrule
$a_{21}>a_{12}$ & 0.8807 & 0.0850 \\
$a_{21}<a_{12}$ & -0.8853 & 0.0844 \\
Leading-order prediction for first row & 0.7590 & -- \\
\bottomrule
\end{tabular}
\caption{Directional cross-area validation using 50,000 exact bivariate Hawkes simulations per row. The sign reverses when the excitation direction is reversed.}
\label{tab:cross-area-num}
\end{table}

\FloatBarrier
\section{Master theorem}
\label{sec:master}

\begin{theorem}[master representation theorem for paths]
\label{thm:master}
The constructions in the paper assemble into the following six-layer representation theorem, each layer being valid under the hypotheses stated in the cited results.
\begin{enumerate}[label=(\roman*),leftmargin=2.25em]
\item \textbf{Algebraic path layer.} Bounded-variation signatures satisfy Chen multiplicativity and the shuffle identity; group-like signatures have Lie logarithms, and reduced bounded-variation paths are characterized by their signatures modulo tree-like equivalence.
\item \textbf{Jump and stream layer.} Finite-variation \cadlag{} paths admit forward \Ito{} and Marcus lifts. The forward lift fails shuffle by the jump bracket, while the Marcus correction restores geometricity. For pure-jump paths, the forward lift is the iterated-sums signature and the Marcus lift is Hoffman's exponential image.
\item \textbf{Geometricity-defect layer.} The two canonical failures of shuffle multiplicativity are second order: quadratic covariation for non-geometric conventions and coordinate covariance for expected signatures. Consequently, an expected signature is group-like only in the deterministic reduced-path case.
\item \textbf{Probabilistic and geometric layer.} Truncated signature kernels yield MMDs equal to Euclidean distances between truncated expected signatures, while the truncations of geometric signatures form a central tower of free nilpotent groups whose forgotten layers have Witt dimensions.
\item \textbf{Self-exciting layer.} Affine processes have finite linear closures for truncated expected signatures after state-weight augmentation. Scalar exponential Hawkes clocks admit the explicit systems \eqref{eq:hawkes-matrix} and \eqref{eq:A2}, local parameter recovery through \eqref{eq:ident-reconstruct}, and a directional cross-area detecting two-channel excitation asymmetry to first order.
\item \textbf{Boundary and continuation layer.} Raw expected signatures have intrinsic heavy-tail moment thresholds, normalized expected signatures provide bounded characteristic replacements under the hypotheses of \citet{chevyrev2022}, reversal acts by the tensor antipode, and signature large deviations follow by contraction from rough-path large deviations.
\end{enumerate}
\end{theorem}

\begin{proof}
Item (i) is \cref{thm:chen,thm:shuffle,thm:logsig,thm:hambly}. Item (ii) is \cref{prop:jump-correction,prop:ito-defect,prop:ito-sums,prop:hoffman-exact,thm:hopf-square}. Item (iii) is \cref{thm:geo-defect,cor:grouplike-det}. Item (iv) is \cref{prop:kernel-mmd,prop:central-tower}. Item (v) is \cref{prop:affine-sig,thm:hawkes-closure,thm:hawkes-A2,thm:hawkes-identifiability,thm:cross-area-leading}. Item (vi) is \cref{thm:stable-threshold,prop:normalized,thm:antipode,cor:area-odd,thm:itolyons,cor:ldp}. The expected-signature characterization used in the probabilistic layer is the standard one under integrability and determinacy assumptions \citep{chevyrev2016,chevyrev2022}.
\end{proof}

\section{Open problems and conjectures}
\label{sec:open}

\begin{conjecture}[signature-native graph recovery]
\label{conj:dag}
For sparse multivariate Hawkes kernels, the sign pattern of the expected cross-areas $\{\E[A_T^{ij}]\}_{i<j}$ over one or several horizons recovers the dominant directed excitation graph after controlling for common baselines and self-excitation.
\end{conjecture}

\begin{enumerate}[label=\textup{(O\arabic*)},leftmargin=2.4em]
\item \textbf{Signature cumulants and Witt compression} \emph{(conditional).} The logarithm of the expected signature is generically not Lie by \cref{cor:grouplike-det}, but the signature cumulants of \citet{bonnier2020signature} satisfy a Magnus-type recursion indexed by a Lyndon basis. Under the affine hypotheses of \cref{prop:affine-sig} this system is finite at each level. Whether it can be solved in $O(\sum_{k\le N}\ell_d(k))$ rather than $O(d^N)$ operations is open.
\item \textbf{Determinism modulus} \emph{(conjectural).} The Carnot distance from $\E\Sig^{(N)}$ to $G^{(N)}$ is conjectured comparable to the total coordinate standard deviation, giving a group-geometric modulus of non-determinism that refines \cref{cor:grouplike-det}.
\item \textbf{Finite-level Hawkes identifiability} \emph{(conjectural).} Level-two expected signatures may identify multivariate exponential Hawkes parameters beyond the local scalar identification in \cref{thm:hawkes-identifiability}. This is a finite-dimensional Jacobian question once the level-two moment system is written explicitly.
\end{enumerate}

\section{Conclusion}
\label{sec:conclusion}

The path is the primary object. The signature represents its chronological algebra; the log-signature removes tensor redundancy; the reduced path group controls parametrization and cancellation; jump lifts specify how discontinuities are interpreted; and expected signatures turn path geometry into law-level coordinates. The Hopf square shows that pure-jump \Ito{} signatures, iterated sums, Marcus chord fills, and Hoffman's exponential are one algebraic mechanism. The Hawkes closure theorem shows that self-exciting jump processes fit this path-first framework without pretending to be Levy processes: their expected signatures are governed by finite triangular ODEs and explicit matrix exponentials at fixed truncation level. The cross-area theorem identifies a genuinely directional level-two coordinate, and the antipode theorem explains its reversal sign. The stable-threshold theorem marks the boundary where raw expected signatures cease to be finite, while normalized signatures and contraction-based large deviations show how the framework continues beyond raw moments. These results give a standalone path theory with original, checkable consequences for event-driven stochastic modeling.

\appendix
\section{Ogata thinning used in the numerical experiment}
\label{app:ogata}

For the exponential Hawkes process, the intensity decreases deterministically between accepted events. Ogata thinning uses the current intensity as a valid upper bound until the next candidate time.

\begin{enumerate}[leftmargin=2.25em]
\item Initialize $t=0$ and $\lambda=\lambda_0$.
\item Draw $E\sim\mathrm{Exp}(\lambda)$ and set $t'=t+E$.
\item Decay the intensity to $\lambda^- = \mu+(\lambda-\mu)e^{-\beta(t'-t)}$.
\item Accept the candidate with probability $\lambda^-/\lambda$.
\item If accepted, record $t'$ and set $\lambda=\lambda^-+\alpha$; otherwise set $\lambda=\lambda^-$.
\item Continue until the candidate exceeds $T$.
\end{enumerate}

Correctness follows from thinning for point processes with predictable intensities \citep{ogata1981,lewis1979,gillespie1976}. For exponential kernels the monotonicity between jumps makes the current intensity a valid dominating rate, and the exact exponential-Hawkes simulation of \citet{dassios2013} provides an independent cross-check.

\section{Claim-tier summary}
\label{app:claims}

The table separates original contributions from classical inputs and cited conditional tools. The word ``proved'' in the paper means either proved directly in the text or explicitly reduced to the cited theorem listed here; the ``Role'' column records which is which.

{\footnotesize
\begin{longtable}{@{}>{\raggedright\arraybackslash}p{0.24\linewidth}>{\raggedright\arraybackslash}p{0.18\linewidth}>{\raggedright\arraybackslash}p{0.14\linewidth}>{\raggedright\arraybackslash}p{0.36\linewidth}@{}}
\toprule
Result & Role & Tier & Basis \\
\midrule
\endfirsthead
\toprule
Result & Role & Tier & Basis \\
\midrule
\endhead
Chen identity and shuffle & Classical input & \Proved & Direct proof and Chen theory; \cref{thm:chen,thm:shuffle} \\
Log-signature Lie property & Classical input & \Proved & Chen--Ree theorem; \cref{thm:logsig} \\
Reduced path uniqueness & Classical input & \Proved & Hambly--Lyons theorem; \cref{thm:hambly} \\
Marcus--\Ito{} jump correction & Classical input / local proof & \Proved & Direct level-two calculation; \cref{prop:jump-correction} \\
\Ito{} shuffle defect & New synthesis, direct proof & \Proved & Finite-variation product rule; \cref{prop:ito-defect} \\
Geometricity defect theorem & New synthesis, direct proof & \Proved & Bracket/covariance identities; \cref{thm:geo-defect} \\
Expected signature leaves the group & New synthesis & \Proved & Covariance defect plus signature uniqueness; \cref{cor:grouplike-det} \\
Kernel MMD decomposition & Supporting consequence & \Proved & Bilinear expansion; \cref{prop:kernel-mmd} \\
Central tower of truncations & Geometric interpretation & \Proved & Nilpotent Lie grading; \cref{prop:central-tower} \\
Marcus uniqueness for monotone completed graphs & Supporting uniqueness statement & \Proved & Hambly--Lyons plus time-augmented monotone graph; \cref{thm:marcus-unique} \\
\Ito{} equals sums; Marcus equals Hoffman(\Ito{}) & New structural theorem & \Proved & Pure-jump recursion and Hoffman exponential; \cref{prop:ito-sums,prop:hoffman-exact,thm:hopf-square} \\
Discrete-to-continuous bound without Hoffman & Supporting estimate & \Proved & Telescoping product estimate; \cref{thm:discrete-bound} \\
Affine expected-signature closure & Conditional framework & \Conditional & Requires non-explosion, finite moments, polynomial-domain closure; \cref{prop:affine-sig} \\
Hawkes expected-signature closure and $A_2$ matrix & New calculation & \Proved & Direct generator calculation; \cref{thm:hawkes-closure,thm:hawkes-A2} \\
Hawkes local identifiability & New calculation & \Proved & Derivative reconstruction; \cref{thm:hawkes-identifiability} \\
Cross-area leading-order law & New calculation & \Proved & First-order expansion around independent Poisson base; \cref{thm:cross-area-leading} \\
Reversal antipode and cross-area oddness & Classical Hopf fact used structurally & \Proved & Tensor antipode; \cref{thm:antipode,cor:area-odd} \\
Cross-area graph recovery & Open direction & \Conjectural & \cref{conj:dag} \\
Stable moment threshold & Limitation theorem & \Proved & Stable moment criterion applied to diagonal signature coordinate; \cref{thm:stable-threshold} \\
Normalized expected signature & Cited continuation & \Conditional & Characteristic normalized kernel hypotheses; \cref{prop:normalized} \\
Signature LDP & Standard consequence & \Proved & Contraction principle and signature continuity; \cref{thm:itolyons,cor:ldp} \\
Expected signature determines law & Cited conditional theorem & \Conditional & Requires integrability and moment determinacy \citep{chevyrev2016,chevyrev2022} \\
\bottomrule
\end{longtable}
}

\bibliographystyle{plainnat}
\bibliography{references}

\end{document}